\documentclass[11pt, reqno]{amsart}
\usepackage{fullpage}

 \newif\ifHideFoot
\HideFoottrue
\HideFootfalse     
\usepackage{amssymb,amsmath,amsthm,amscd,mathrsfs,graphicx, color}
\usepackage[cmtip,all,matrix,arrow,tips,curve]{xy}
 \usepackage{hyperref}
\usepackage[usenames,dvipsnames]{xcolor}
\usepackage{xypic}
\hypersetup{colorlinks=true,citecolor=ForestGreen,linkcolor=Maroon,urlcolor=NavyBlue}
 \usepackage{mathtools}
\usepackage{mathpazo}

\usepackage[normalem]{ulem}

\usepackage{cleveref}

\usepackage{enumitem}

\numberwithin{equation}{section}
\newtheorem{teo}{Theorem}[section]
\newtheorem{pro}[teo]{Proposition}
\newtheorem{lem}[teo]{Lemma}
\newtheorem{cor}[teo]{Corollary}

\newtheorem{teoalpha}{Theorem}

\theoremstyle{definition}
\newtheorem{exa}[teo]{Example}

\theoremstyle{remark}
\newtheorem{rem}[teo]{Remark}

\ifHideFoot

\newcommand{\Yano}[1]{}
\newcommand{\Jeff}[1]{}
\newcommand{\Jonathan}[1]{}

\else 

\newcommand{\marg}[1]{\normalsize{{
\color{red}\footnote{{\color{blue}#1}}}{\marginpar[\vskip
-.25cm{\color{red}\hfill\tiny\thefootnote$\implies$}]{\vskip
-.2cm{\color{red}$\impliedby$\tiny\thefootnote}}}}}
\newcommand{\Yano}[1]{\marg{(Yano) #1}}
\newcommand{\Jeff}[1]{\marg{(Jeff) #1}}
\newcommand{\Jonathan}[1]{\marg{(Jonathan) #1}}

\fi

\global\let\hom\undefined
\DeclareMathOperator{\hom}{Hom}
\def\ext{\operatorname{Ext}}

\def\et{{\rm \acute et}}

\def\mmu{{\pmb\mu}}
\def\aalpha{{\pmb\alpha}}

\def\proj{{\mathbb P}}

\def\ff{{\mathbb F}}
\def\gp{{\mathbb G}}

\def\integ{{\mathbb Z}}

\def\iso{\cong}

\renewcommand{\bar}[1]{{\overline{#1}}}

\DeclareMathOperator{\Hom}{Hom}

\DeclareMathOperator{\spec}{Spec}

\newcommand{\powser}[1]{[\![#1]\!]}

\def\ab{\mathrm{ab}}
\def\et{\mathrm{et}}

\newcommand{\st}[1]{\left\{#1\right\}}

 \mathtoolsset{centercolon}

\title{Images of abelian schemes}

\author{Jeffrey D. Achter}
\address{Colorado State University, Department of Mathematics,
Fort Collins, CO 80523,
USA}
\email{j.achter@colostate.edu}

\author{Sebastian Casalaina-Martin}
\address{University of Colorado, Department of Mathematics, 
Boulder, CO 80309, USA }
\email{casa@math.colorado.edu}

\author{Jonathan Wise}
\address{University of Colorado, Department of Mathematics, 
Boulder, CO 80309, USA }
\email{jonathan.wise@math.colorado.edu}

\thanks{
  The authors were partially supported by respective Simons Foundation grants 637075, 581058, and 636210.}

\date{\today}

\begin{document}

\maketitle

\begin{abstract}
  We provide some conditions for the image of a morphism of abelian schemes to again be an abelian scheme.  For context, in characteristic 0, the image is always an abelian scheme; in mixed and positive characteristic the image can fail to be an abelian scheme, and so it is in this setting that the conditions we provide are pertinent.     
\end{abstract}

Let $S$ be a connected, locally Noetherian scheme, and let $X$ and $Y$
be abelian schemes over $S$.  A homomorphism of abelian schemes $X \to Y$
is an $S$-morphism of schemes which carries the identity section of
$X$ to that of $Y$.  It has long been known that, thanks to rigidity,
such a morphism necessarily respects the group structure
\cite[Prop.~6.1]{GIT}.  We are interested in when the 
\emph{image} of such a morphism is again an abelian scheme:

\begin{teoalpha}[Images of abelian schemes]\label{T:ImIsAbSch}
	Let $S$ be a locally noetherian scheme.
  Let $f\colon X \to Y$ be a homomorphism of abelian schemes over $S$.
  The following are equivalent:
  \begin{enumerate}[label=(\alph*)]
 \item \label{T:ImIsAbSch1}   $\ker(f)$ is a flat sub-$S$-group scheme of $X/S$, 
 \item \label{T:ImIsAbSch2} the scheme theoretic image $f(X)\subseteq Y$ is a sub-abelian scheme over $S$, and  
 \item \label{T:ImIsAbSch3}  $f$
  factors as
 \begin{equation}\label{E:ImIsAbSch}
\xymatrix{
X\ar@{->>}[r]_\pi \ar@/^1pc/[rr]^f&Z \ar@{^(->}[r]_\iota&Y
}
\end{equation}
where $X/\ker (f)\cong Z\subseteq Y$ is a sub-$S$-group scheme,
$\pi$ is a flat surjective homomorphism, and $\iota$ is the natural
inclusion.  Moreover, $Z=f(X)$, and the factorization \eqref{E:ImIsAbSch} is stable under base change.
\end{enumerate}

Moreover, if $S$ is of characteristic $0$, or $S=\operatorname{Spec}K$ for a field $K$, 
or  if every local ring of $S$ is either of characteristic zero or is 
a discrete valuation ring  of mixed characteristic with valuation $\nu$, residue field of characteristic $p>0$, and ramification index $e:=\nu(p)$ such that $e<p-1$, 
then \ref{T:ImIsAbSch1},  \ref{T:ImIsAbSch2}, and  \ref{T:ImIsAbSch3} hold.
\end{teoalpha}

Certain cases of \Cref{T:ImIsAbSch} seem to be well-known in the
literature, in particular, the case of abelian varieties (in other words,
the case of the theorem where $S$ is the spectrum of
a field).
Since we will use the special case of \Cref{T:ImIsAbSch} over a field
in various places, we will review a short
proof of that case below (see \Cref{P:FlatSuff}).

By way of a modest caution, we note that certain plausible-sounding
extensions of \Cref{T:ImIsAbSch} fail.  
For instance, some conditions are needed on $S$, as there are morphisms  of abelian schemes whose images are not abelian schemes (see \Cref{E:Serre} and \Cref{E:charp}).  
In these examples, the image is a flat proper group scheme (see \Cref{R:FlatDVR}) that is not an abelian scheme; in other words, while it is sufficient that the kernel of a homomorphism  be flat to imply that the image is an abelian scheme (\Cref{T:ImIsAbSch}\ref{T:ImIsAbSch1}), it is not sufficient that the image of the homomorphism be flat (cf.~\Cref{T:ImIsAbSchArt} over Artinian local rings for contrast).
In addition, while \Cref{T:ImIsAbSch}\ref{T:ImIsAbSch1} states that if the image of a morphism of abelian schemes is an abelian scheme, then the image is stable under base change, in \Cref{R:BC} and \Cref{R:ImNotStable} we give examples of morphisms of abelian schemes where the image is not stable under base change.
Finally, the \emph{kernel} of a morphism of abelian schemes need not be an abelian scheme.  Indeed, 
even in the case of abelian varieties, the kernel of a homomorphism of abelian varieties may be disconnected,  and in positive characteristic may be non-reduced.

The following  criterion for the image of a homomorphism of an abelian scheme to be an abelian scheme turns out to be useful for us elsewhere:

                  \begin{teoalpha}\label{T:NecSuff}
Let $f:X\to Y$ be a homomorphism of abelian schemes over $S$.  

\begin{enumerate}[label=(\alph*)]
\item \label{T:NecSuff1} The schematic image $f(X)$ is an abelian scheme  if and only if  for 
every  Artinian local ring $(R,\mathfrak m)$ with algebraically  closed residue field and every morphism $S':= \operatorname{Spec} R\to S$, the 
 base change $f_{S'}:X_{S'}\to Y_{S'}$ has the property that the scheme theoretic image 
 $f_{S'}(X_{S'})\subseteq Y_{S'}$ contains an abelian subscheme 
of the same dimension as $f_{S'}(X_{S'})$.

 \item \label{T:NecSuff2} If $S$ is reduced,
 then the  schematic image $f(X)$ is an abelian scheme  
  if and only if  for 
every  DVR $R$ and  every morphism $S':= \operatorname{Spec} R\to S$, the 
 base change $f_{S'}:X_{S'}\to Y_{S'}$ has the property that the scheme theoretic image $f_{S'}(X_{S'})\subseteq Y_{S'}$ contains an abelian subscheme whose generic fiber is of the same dimension as the generic fiber of $f_{S'}(X_{S'})$.   
\end{enumerate}
\end{teoalpha}

There is also the following variation on the theme: 

\begin{teoalpha}
  \label{T:containsabenough}
  Let $f:X \to Y$ be a homomorphism of abelian schemes over $S$, with $S$ connected.
  \begin{enumerate}[label=(\alph*)]
  \item   \label{T:containsabenougha} The schematic image $f(X)$ has constant relative dimension over $S$.
  \item   \label{T:containsabenoughb}The image $f(X)$ is an abelian scheme if and only if it contains an abelian scheme of the same relative dimension.
  \end{enumerate}
\end{teoalpha}

We note here that other conditions for the image of a homomorphism of abelian schemes to be an abelian scheme can be found in \cite[\S4]{FS08}.   
For instance, over a DVR, they show that it suffices to show that the image is normal (note that in this case the image is always a flat proper group scheme; see \Cref{R:FlatDVR}).  They also point out the fact that over an arbitrary base, it suffices to show that the image is flat and every geometric fiber is smooth (see \Cref{R:DVRconds}).

This paper is organized as follows.  In \S \ref{S:ImClPre}, we review some preliminaries on morphisms of abelian schemes, which allow us to establish the equivalence of conditions \ref{T:ImIsAbSch1}, \ref{T:ImIsAbSch2}, and \ref{T:ImIsAbSch3} of \Cref{T:ImIsAbSch} in \Cref{P:FlatSuff} (which immediately implies \Cref{T:ImIsAbSch} in the case where $S$ is the spectrum of a field), and also to reduce \Cref{T:ImIsAbSch} in general to the case where $S$ is the spectrum of an Artinian local ring (\Cref{L:R2Artin}).   
In \S \ref{S:MainPfs}
 we use these results to prove \Cref{T:ImIsAbSch},  \Cref{T:NecSuff}, and \Cref{T:containsabenough}, except for  the assertion in  \Cref{T:ImIsAbSch} that  over \emph{non-reduced bases} in characteristic $0$, 
the equivalent conditions  \ref{T:ImIsAbSch1},  \ref{T:ImIsAbSch2}, and  \ref{T:ImIsAbSch3} of \Cref{T:ImIsAbSch} hold.   That last assertion is proved in \S \ref{S:KerMorph} via deformation theory, as a consequence of  \Cref{T:KerMorph}.  In \S\ref{S:Examples}, we give examples of morphisms of abelian schemes with image that is not an abelian scheme (\Cref{E:Serre} and \Cref{E:charp}); the example in mixed characteristic is due to Serre, and our example in pure characteristic is a variation on the theme.  Finally, in \S \ref{S:ArtinRing}, we show that over \emph{Artinian local rings}, the image of a morphism of abelian schemes is an abelian scheme if and only if the image is flat (\Cref{T:ImIsAbSchArt}); recall that this does not hold in general (e.g.,  \Cref{E:Serre} and \Cref{E:charp}).
In \Cref{S:4} we prove a vanishing of an Ext group (\Cref{P:2}) that we need for a deformation theory argument in \Cref{S:KerMorph}.

\subsection*{Acknowledgments}

We thank Brian Conrad and Aaron Landesman for useful comments on a draft. We also want to especially thank Patrick Brosnan for pointing out a serious error in a previous version of this paper.

\section{Preliminaries}\label{S:ImClPre}

A crucial step 
     in proving \Cref{T:ImIsAbSch} is  the standard fact that the quotient of an abelian scheme
by a flat kernel of a homomorphism is well-behaved; we state here a slightly more general version of this standard fact, which we will also use elsewhere:

\begin{teo}[{\cite[Exp.~V, Cor.~10.1.3]{sga3-1}}]\label{T:quot}
Let $f\colon X\to Y$ be a homomorphism of $S$-group schemes, with $X$ locally of finite presentation over $S$.  If $\ker(f)$ is flat over $S$, then the fppf sheafification of the functor to groups $(S'\to S)\mapsto X(S')/(\ker(f)(S'))$ is representable by an $S$-group  scheme $X/\ker (f)$ locally of finite presentation over $S$, and $f$ factors as 
$$
\xymatrix{
X\ar[rr]^f \ar[rd]_{\pi}&& Y \\
& X/\ker (f) \ar[ru]_\iota& 
}
$$
where $\pi$ is the canonical projection and $\iota$ is a monomorphism.
Moreover, this factorization is stable under base change.  If in addition $X/S$ is proper and $Y/S$ is separated, then $X/\ker(f)$ is proper over $S$, $\iota$ is a closed immersion, and $X/\ker(f)$ agrees with $f(X)$,  the scheme theoretic image of $X$ under $f$.
\end{teo}

\begin{proof}
The assertion of the representability of the fppf sheafififcation of the functor to groups $(S'\to S)\mapsto X(S')/(\ker(f)(S'))$, as well as the factorization of $f$ into the canonical projection $\pi$ and the monomorphism $\iota$, is exactly \cite[Exp.~V, Cor.~10.1.3]{sga3-1}.  The fact that this factorization is stable under base change follows directly from the fact that the functor $(S'\to S)\mapsto X(S')/(\ker(f)(S'))$ is stable under base change by construction.

Now assume that $X/S$ is proper and $Y/S$ is separated.  Since every
monomorphism of schemes is separated (e.g., \cite[Prop.~9.13(1)]{GW})
we can conclude by stability of separatedness under composition that
$X/\ker (f)$ is separated.  Next we claim that the canonical
projection $\pi$ is surjective. Indeed, as kernels are stable under base
change by construction, for every $S'\to S$ the
homomorphism of groups $X(S')\to X(S')/\ker(f)(S')$ is surjective.
Since a surjective homorphism of functors remains surjective after
passing to the fppf sheafification, $\pi$ is surjective as
claimed.  Consequently, we may conclude that $X/\ker(f)$ is proper
(e.g., \cite[Prop.~12.59]{GW}).  It follows that $\iota$ is proper
(e.g., \cite[Prop.~12.58]{GW}), and therefore a closed immersion (e.g., \cite[\href{https://stacks.math.columbia.edu/tag/04XV}{Lem.~04XV}]{stacks-project}).  The
agreement of $X/\ker(f)$ with $f(X)$ is then clear.
\end{proof}

\begin{pro}\label{P:FlatSuff}
  Let $f\colon X \to Y$ be a homomorphism of group schemes over $S$ with $X/S$ smooth and proper and $Y/S$ separated.
 Then the following are equivalent:
  \begin{enumerate}[label=(\alph*)]
 \item \label{P:FlatSuff1}   $\ker(f)$ is a flat sub-$S$-group scheme of $X/S$, 
 \item \label{P:FlatSuff2} the scheme theoretic image $f(X)\subseteq Y$ is a smooth proper sub-$S$-group scheme, and,  
 \item \label{P:FlatSuff3} $f$
  factors as
 \begin{equation}\label{E:ImIsGpSch}
\xymatrix{
X\ar@{->>}[r]_\pi \ar@/^1pc/[rr]^f&Z \ar@{^(->}[r]_\iota&Y
}
\end{equation}
where $X/\ker (f)\cong Z\subseteq Y$ is a sub-$S$-group scheme,
$\pi$ is a flat surjective homomorphism, and $\iota$ is the natural
inclusion.  Moreover, $Z=f(X)$, and the factorization \eqref{E:ImIsGpSch} is stable under base change.
\end{enumerate}
Moreover, if $S=\operatorname{Spec}K$ for a field $K$, then \ref{P:FlatSuff1}, \ref{P:FlatSuff2}, and \ref{P:FlatSuff3} hold.
\end{pro}

\begin{proof} 
First we show that \ref{P:FlatSuff1}, \ref{P:FlatSuff2}, and \ref{P:FlatSuff3} hold in the case where $S=\operatorname{Spec} K$.
Since we are working over a field, assertion \ref{P:FlatSuff1}  holds trivially.  
Consequently, from \Cref{T:quot} one can factor $f$ through the proper closed sub-group scheme $Z=X /\ker (f)$.  
Next we show that   $Z$ is smooth.  Initially, note that $Z$, as the scheme theoretic image of $f$, is reduced, 
since $X$ is.  
In fact, since the factorization in \Cref{T:quot} is stable under base change, we can use the same argument to show that $Z$ is geometrically reduced.
Therefore, $Z$  is generically smooth, and, being a group scheme, it is consequently smooth.  Therefore, we have that $Z$ is a smooth proper group scheme.
  The flatness of $\pi$ follows from say \cite[Lem.~6.12]{GIT}, which implies that any surjective morphism of group schemes with flat source and smooth target is flat.  This shows that  \ref{P:FlatSuff1}, \ref{P:FlatSuff2}, and \ref{P:FlatSuff3}  hold for $S=\operatorname{Spec}K$.  

With this, we now show the equivalence of  \ref{P:FlatSuff1}, \ref{P:FlatSuff2}, and \ref{P:FlatSuff3} over a general $S$.  
Given assertion  \ref{P:FlatSuff3} in \Cref{T:ImIsAbSch}, then assertion  \ref{P:FlatSuff1} in the theorem is obvious since $\pi$ is assumed to be flat, and the kernel is obtained via base change; assertion  \ref{P:FlatSuff2} holds as it is part of assertion   \ref{P:FlatSuff3}.   Similarly, given assertion  \ref{P:FlatSuff2}, then assertion  \ref{P:FlatSuff1} holds by virtue of  \cite[Lem.~6.12]{GIT}, which asserts that $\pi$ is flat.  

Therefore, let us consider the implication  \Cref{P:FlatSuff} \ref{P:FlatSuff1} $\implies$ \Cref{P:FlatSuff}\ref{P:FlatSuff3}.  
Putting together \Cref{T:quot} and \Cref{P:FlatSuff}\ref{P:FlatSuff1}, we obtain a factorization 
$$
\xymatrix{
X\ar@{->>}[r]_\pi \ar@/^1pc/[rr]^f&Z \ar@{^(->}[r]_\iota&Y
}
$$
where $X/\ker(f)\cong Z\subseteq Y$ is a commutative sub-$S$-group scheme, proper over $S$,  $\pi$ is a  surjective homomorphism, and $\iota$ is the natural inclusion.  Moreover, this factorization is stable under base change.

We need to show that $Z$ is smooth and proper. First we prove that $\pi$ is flat.  Since $X$ is flat over $S$, it suffices to show that $\pi_s\colon X_s\to Z_s$ is flat for all $s\in S$ (e.g., \cite[Cor.~14.27]{GW}).  Since the construction of $Z$ is stable under base change, we have that    $\pi_s\colon X_s\to Z_s$ agrees with the map $X_s\to X_s/\ker (f_s)$, and this is  a flat surjective homomorphism to a smooth proper group scheme, as it is the image of the homomorphism  $\pi_s\colon X_s\to Y_s$  (the case we just proved).    Thus $\pi$ is flat.  Being surjective, $\pi$ is faithfully flat, and therefore, $X$ being flat over $S$, we may conclude that $Z$ is flat over $S$ (e.g., \cite[Cor.~14.12]{GW}). 
 Now that $Z/S$ is flat (and locally of finite presentation, by virtue of being proper), we can check the smoothness of $Z$ over $S$ by checking that the geometric fibers are smooth; since we have seen that $Z_s$ is a smooth proper group scheme, we have that the fibers are smooth. 
 Therefore $Z/S$ is smooth and proper, and we are done.
\end{proof}

\Cref{P:FlatSuff} shows that conditions \ref{T:ImIsAbSch1}, \ref{T:ImIsAbSch2}, and \ref{T:ImIsAbSch3} of \Cref{T:ImIsAbSch} are equivalent.  Since \Cref{T:ImIsAbSch}~\ref{T:ImIsAbSch1} is automatic when $S$ is the spectrum of a field, this proves~\Cref{T:ImIsAbSch} when $S$ is the spectrum of a field.

\subsection{Reduction to Artinian rings and DVRs} 

By virtue of \Cref{P:FlatSuff}, in order to prove \Cref{T:ImIsAbSch}, it suffices to show that the kernel of a given morphism of abelian schemes is flat.  
 For clarity, we recall in the remark below the standard approach to  reducing flatness computations to computations over Artinian local rings with algebraically closed residue field, or, in the case where $S$ is reduced, to the case of DVRs.

\begin{lem}[Reduction to Artinian rings]\label{L:R2Artin}
Let $f:X\to Y$ be a homomorphism of abelian schemes over $S$.  
Then the equivalent conditions of 
\Cref{T:ImIsAbSch}~\ref{T:ImIsAbSch1},  \ref{T:ImIsAbSch2}, and  \ref{T:ImIsAbSch3} hold for $f$ if and only if  for 
every  Artinian local ring $(R,\mathfrak m)$ with algebraically  closed residue field and every morphism $S':= \operatorname{Spec} R\to S$, the equivalent conditions of  
\Cref{T:ImIsAbSch}~\ref{T:ImIsAbSch1},  \ref{T:ImIsAbSch2}, and  \ref{T:ImIsAbSch3} hold for the 
 base change $f_{S'}:X_{S'}\to Y_{S'}$.
  \end{lem}

\begin{proof}
By \Cref{P:FlatSuff}, the forward implication is clear.  For the reverse implication, it suffices to establish condition \Cref{T:ImIsAbSch}\ref{T:ImIsAbSch1}, that $\ker(f)$ flat.  As 
flatness satisfies faithfully flat descent (e.g., \cite[Cor.~14.12]{GW}), to show that $\ker(f)$ is flat over $S$  we may reduce to the case that $S=\operatorname{Spec}R$ for $R$ a local ring, or even a complete local ring.  
Via the strict Henselization, one may then assume that $R$ is a local ring with separably closed residue field. 
In fact, taking a faithfully flat base change, we may assume that $R$ is a complete local ring with algebraically closed residue field  \cite[EGA III Ch.~0, \S 10, Prop.~10.3.1, p.364]{EGAIII}.  Consequently (e.g., \cite[Thm.~B.51]{GW}), it suffices to consider the case where $(R,\mathfrak m)$ is an Artinian local ring with algebraically  closed residue field.  
\end{proof}

\begin{lem}[Reduction to DVRs]\label{L:R2DVR}
Let $f:X\to Y$ be a homomorphism of abelian schemes over $S$ with $S$ reduced.  
Then the equivalent conditions of 
\Cref{T:ImIsAbSch}\ref{T:ImIsAbSch1},  \ref{T:ImIsAbSch2}, and  \ref{T:ImIsAbSch3} hold if and only if  for 
every  DVR  $R$ and every morphism $S':= \operatorname{Spec} R\to S$, 
 the equivalent conditions of  
\Cref{T:ImIsAbSch}\ref{T:ImIsAbSch1},  \ref{T:ImIsAbSch2}, and  \ref{T:ImIsAbSch3} hold for the 
 base change $f_{S'}:X_{S'}\to Y_{S'}$.
\end{lem}

\begin{proof}
By \Cref{P:FlatSuff}, the forward implication is clear.  For the reverse implication, it suffices to establish condtion \Cref{T:ImIsAbSch}\ref{T:ImIsAbSch1}, that $\ker(f)$ flat. 
For this we can use the valuative criterion for flatness (e.g., \cite[Thm.~14.34]{GW}), which implies that it suffices to show that  for 
every  DVR  $R$ and every morphism $S':= \operatorname{Spec} R\to S$, the base change of $\ker(f)$ is flat. 
Since kernels are stable under base change,  it suffices to show that the kernel of the 
 base change $f_{S'}:X_{S'}\to Y_{S'}$ is flat, but this is true by assumption. 
\end{proof}

\section{Proofs of \Cref{T:ImIsAbSch} over reduced bases, \Cref{T:NecSuff}, and \Cref{T:containsabenough}}\label{S:MainPfs}

In this section we prove \Cref{T:ImIsAbSch},  \Cref{T:NecSuff}, and \Cref{T:containsabenough}, except for  the assertion in  \Cref{T:ImIsAbSch} that  over \emph{non-reduced bases} in characteristic $0$, 
the equivalent conditions  \ref{T:ImIsAbSch1},  \ref{T:ImIsAbSch2}, and  \ref{T:ImIsAbSch3} of \Cref{T:ImIsAbSch} hold.  We postpone that proof until the next section.

\begin{proof}[Proof of \Cref{T:ImIsAbSch}, Part I]
The equivalence of the conditions 
\Cref{T:ImIsAbSch}\ref{T:ImIsAbSch1},  \ref{T:ImIsAbSch2}, and  \ref{T:ImIsAbSch3}, as well as the case $S=\operatorname{Spec}K$ for a field $K$, are contained in \Cref{P:FlatSuff}.  

The case where $S$ is reduced of characteristic $0$ is as follows.  
Note that we give a proof in characteristic $0$ in the case where $S$ is not assumed to be reduced later, via \Cref{T:KerMorph}, but we find the proof that follows,  in the reduced case, to be instructive, particularly in describing the failure of images of homomorphisms of abelian schemes to be abelian schemes over DVRs in positive and mixed characteristic. 

Continuing with the proof,  by \Cref{L:R2DVR}, we have reduced to the case where $S=\operatorname{Spec}R $ for a DVR $R$.  Again by \Cref{L:R2DVR}, it suffices to show that the image of $f$ is an abelian scheme.  
We use the Ner\'on--Ogg--Shafarevich criterion (e.g., \cite[Thm.~5, p.183]{BLR}, \cite[Thm.~1]{serretate}) to extend the image.  More precisely, we restrict to the generic point $\eta$ of $S$, and observe that over a field we have that $f_\eta$ factors as $X_\eta\twoheadrightarrow Z_\eta \hookrightarrow Y_\eta$ for some abelian variety $Z_\eta$ (\Cref{P:FlatSuff}).  The  Ner\'on--Ogg--Shafarevich criterion then implies that $Z_\eta$ extends to an abelian scheme $Z$ over $S$; \emph{the fact that we are in characteristic $0$} implies there is a containment $Z\subseteq Y$ \cite[Prop.~2, \S 7.5, p.186]{BLR} (note this can fail outside of characteristic $0$; see   \cite[Exa.~8, \S 7.5, p.190]{BLR} and \S \ref{S:Examples}).  
  Since $Z$ is closed, the morphism $f$ factors through $Z$, and since $f:X\to Z$ is generically surjective, it is surjective ($f:X\to Y$ is proper).  Thus we have factored $f$ as $X\twoheadrightarrow Z\hookrightarrow Y$.

  Finally, consider the case where every local ring is a discrete valuation ring of mixed characteristic with valuation $\nu$, residue field of characteristic $p>0$, and ramification index $e:=\nu(p)$ such that $e<p-1$.  It suffices to consider the local case.  The argument is the same as the previous case, except that we use \cite[Thm.~4(i), \S 7.5, p.187]{BLR} to get the inclusion of $Z$ into $Y$.

 \end{proof}

\begin{proof}[Proof of \Cref{T:NecSuff}]
For \ref{T:NecSuff2}, we argue as follows.  By \Cref{L:R2DVR}, we can reduce to the case that $S=\operatorname{Spec}R$ for a DVR.
So suppose that there is an abelian subscheme $Z\subseteq Y$ with $Z$ contained in $f(X)$ (i.e., we have closed embeddings $Z\hookrightarrow  f(X)\hookrightarrow Y$),  and that 
  for the generic point $\eta$ of $S$, we have  $\dim Z_\eta = \dim f(X)_\eta$.
 As $Z_\eta\subseteq f(X)_\eta$ are both proper group schemes of  the same dimension, we can conclude that $Z_\eta$ is the reduced subscheme of the identity component of $f(X)_\eta$.
    If we consider the composition $X_\eta \to f(X)_\eta \to Y_\eta$, then since $X_\eta$ is reduced and connected, the morphism $X_\eta \to f(X)_\eta$ must factor through $Z_\eta$, as it must factor through the reduction, and takes the identity component to the identity component. 
  Since $S$ is regular and $X/S$ is smooth, $X$ is regular, and therefore the rational map $X\dashrightarrow Z$ extends to a morphism \cite[Cor.~6, \S 8.4, p.234]{BLR}.  By the universal property of the scheme theoretic image, this implies that $f(X)=Z$. 
  
  \medskip 
 For \ref{T:NecSuff1}, we argue as follows. 
 By \Cref{L:R2Artin}, we can reduce to the case that $S=\operatorname{Spec}R$ for an Artinian local ring $(R,\mathfrak m)$ with algebraically  closed residue field.
So suppose that there is an abelian subscheme $Z\subseteq Y$ with $Z$
contained  in $f(X)$ (i.e., we have closed embeddings
$Z\hookrightarrow  f(X)\hookrightarrow Y$), such that $\dim Z = \dim
f(X)$.  Since we are working over an Artinian ring, the supports of the
schemes satisfy  $|Z_s|=|Z|\subseteq |f(X)|=|f(X)_s|$, where $s$ is
the special point of $S$.  Now $f(X)$, being the continuous image of
an irreducible space, is irreducible, and so
       the hypothesis that $\dim Z = \dim f(X)$ implies that $|Z|=|f(X)|$, so we can conclude that $Z_s$ is the reduced subscheme of $f(X)_s$.  
Consequently, the morphism $f_{s}:X_{s}\to Y_{s}$ factors through  $Z_{s}$.

  Now choose  a prime $\ell$ that is invertible in $R$, and consider the closed subschemes $X[\ell^n]\hookrightarrow X$ and $Y[\ell^n]\hookrightarrow Y$ for all $n$ .  
It is a basic fact  for abelian schemes (e.g.,  \cite[Proof of Thm.~3.19, p.54]{conradtrace}) that 
$$
X = \overline {\bigcup X [\ell^n]}.
$$ 
  Moreover, from our choice of $\ell$, the $X [\ell^n]$   are proper \'etale group schemes over $S$, and since $R$ is an Artinian local ring with algebraically closed residue field, each of 
$X [\ell^n]$ and $Y[\ell^n]$ is a disjoint union of irreducible
components canonically isomorphic to $S$.  The restricted morphism
$f[\ell^n]\colon X [\ell^n]\to Y[\ell^n]$ is a morphism over $S$.
Therefore, on each irreducible component, $f[\ell^n]$ is an
isomorphism onto its image.  (Of course, some components of
$X[\ell^n]$ may map to the same component of $Y[\ell^n]$.)

Restricting to the special point, we have a factorization
$f_s[\ell^n]:X_s[\ell^n] \to Z_s[\ell^n] \hookrightarrow Y_s[\ell^n]$.
As we have inclusions $Z[\ell^n]\hookrightarrow Y[\ell^n]$, and
$f[\ell^n]$ is an isomorphism on each component, we see that we have a factorization
$f[\ell^n]:X[\ell^n] \to Z[\ell^n] \hookrightarrow Y[\ell^n]$.

Since scheme theoretic images and closures commute (the scheme theoretic closure is the scheme theoretic image of the morphism from the disjoint union of the closed subschemes),  we have 
\begin{equation}\label{E:pr:f(X)1}
Z\subseteq f(X )= f(\overline {\bigcup X [\ell^n]})= \overline {\bigcup f(X [\ell^n])} \subseteq \overline {\bigcup Z [\ell^n]} = Z
\end{equation}
so that $Z=f(X)$.
\end{proof}

\begin{proof}[Proof of Theorem \ref{T:containsabenough}]
We start by proving   \Cref{T:containsabenough}~\ref{T:containsabenougha}.  The group schemes $\ker(f)$ and $f(X)$ are proper group schemes over
$S$ (the former is obtained by base change and the latter is a closed subscheme of the proper $Y$).  
We then consider the short exact sequence of proper, but possibly not flat, group schemes
$$
  \xymatrix{
0\ar[r] &\ker (f)\ar[r] &X\ar[r] &f(X)\ar[r]& 0.}
$$
Since kernels are stable under base change (they are defined via fibered products), for each point $s$ of $S$ we have a short exact sequence
\begin{equation}\label{E:ses-fiber-C}
    \xymatrix{0\ar[r]& \ker (f)_s=\ker (f_s)\ar[r]& X_s\ar@{->>}[r]&  f_s(X_s)\ar[r]& 0.}
\end{equation}
At the same time, we have an inclusion of closed subschemes $f_s(X_s)\subseteq f(X)_s$ (e.g., \cite[p.216]{EH00}), and,  since $f$ is proper, the support of the two schemes is the same (e.g., \cite[p.218]{EH00}). Consequently, $\dim f_s(X_s)=\dim f(X)_s$.    
Combining this with the short exact sequence of group schemes \eqref{E:ses-fiber-C}, we obtain 
$$
\dim \ker (f)_s+\dim f(X)_s= \dim X_s.
$$
Now, by virtue of the fact that the dimension of the fibers of a proper scheme over $S$   
gives an  upper semicontinuous function on $S$ (e.g.,  \cite[\href{https://stacks.math.columbia.edu/tag/0D4I}{Lemma 0D4I}]{stacks-project}), and the fact that  $\dim X_s$ is constant since $X$ is an abelian scheme,  the dimensions of both $\ker(f)_s$ and $f(X)_s$ are constant. In particular, the  schematic image $f(X)$ has constant relative dimension $d$ over $S$ for some integer $d$.
    
 \Cref{T:containsabenough}\ref{T:containsabenoughb} now follows immediately from  \Cref{T:containsabenough}\ref{T:containsabenougha} and \Cref{T:NecSuff}\ref{T:NecSuff1}.
                                \end{proof}

\section{Characteristic $0$ via deformation theory}\label{S:KerMorph}

In \Cref{S:MainPfs}, we proved \Cref{T:ImIsAbSch} over reduced base
schemes $S$.  In this section, we will complete the proof of
\Cref{T:ImIsAbSch} by showing it also holds over non-reduced bases in
characteristic~$0$.  Our proof is based on the following theorem, which is valid in any characteristic.

\begin{teo}\label{T:KerMorph}
	Let $S$ be a scheme and let 
$
f:X\to Y
$
be a morphism of abelian schemes over $S$ with kernel $\ker(f)$. 
       	If for every geometric point $s$ of $S$ we have that $\ker(f)_s$ is smooth and $\Hom(\ker(f)_s, \mathbb G_a) = 0$, where $\ker(f)_s$ is the fiber over $s$,  then $\ker (f)$ is smooth over $S$.
\end{teo}

\begin{rem} \label{R:1}
	Without the smoothness assumption on $\ker(f)_s$, it is still possible to show that $\ker(f)$ is flat (see \Cref{A:Strong}).  However, the proof is somewhat more technical in that generality and the statement is not necessary for the proof of \Cref{T:ImIsAbSch}.
\end{rem}

The converse of \Cref{T:KerMorph} is false: for example, if $f: E\to E'$ the quotient of an elliptic curve over a field $K$ of characteristic $p$ by a subgroup isomorphic to $\mathbb Z/p\mathbb Z$, one has that $\ker (f)$ is smooth but $\operatorname{Hom}(\ker(f),\mathbb G_a)\ne 0$.  Nevertheless, in general the hypothesis that $\Hom(\ker(f)_s, \mathbb G_a)$ vanish cannot be removed.  Examples~\ref{E:Serre} and~\ref{E:charp} show that \Cref{T:KerMorph} would be false without it.

The following proposition helps to explain when the hypotheses of \Cref{T:KerMorph} hold:

\begin{pro} \label{P:5}
	Let $G$ be a proper, commutative group scheme over an algebraically closed field $k$.  Then $H = \hom(G, \mathbb G_m)$ is reduced if and only if $\hom(G, \mathbb G_a) = 0$.
\end{pro}
\begin{proof}
Let $G^{\ab}$ be the reduced structure on the connected component of the identity in $G$, which is a subgroup scheme because $k$ is algebraically closed.  Since $G^{\ab}$ is proper, reduced, and connected, and since $\mathbb G_m$ is affine, every homomorphism $G \to \mathbb G_m$ factors uniquely through $G/G^{\ab}$.  But $G/G^{\ab}$ is of finite type and has finitely many $k$-points, hence is finite (as in \cite[Remark~2.7.3~(v)]{Brion}), so $H$ is the Cartier dual of $G/G^{\ab}$.

	Since $\mathbb G_a$ is affine, every homomorphism $G \to \mathbb G_a$ also factors uniquely through $G/G^{\ab}$.  But the scheme of homomorphisms $G/G^{\ab}\to \mathbb G_a$ may be identified with the Lie algebra of $H$ \cite[\S14, p.~138]{mumfordAV}, which is trivial if (and only if) $H$ is reduced.
\end{proof}

Granting \Cref{T:KerMorph}, we can complete the proof of \Cref{T:ImIsAbSch}:

\begin{proof}[Proof of \Cref{T:ImIsAbSch}, Part II]
All that remains is the case where $S$ is characteristic $0$, but non-reduced.  We conclude in this case using \Cref{T:KerMorph} and \Cref{P:5}, as all group schemes over a field of characteristic zero are smooth \cite[\href{https://stacks.math.columbia.edu/tag/047N}{Lem.~047N}]{stacks-project}, hence reduced.
\end{proof}

\subsection{Deformation theory of commutative group schemes, after Illusie}\label{S:Prelim-deform}  

Here we review some of the deformation theory of commutative group schemes following \cite{illusie2}.  The summary of Illusie's results in \S\ref{S:8} and the technical material of \S\ref{S:9} and \S\ref{S:Site-Change} will be necessary for the proof of \Cref{T:KerMorph} in \S\ref{S:Pf1}.
 
\subsubsection{The co-Lie complex} \label{S:7}

Recall that given a group scheme $G$ over $S$ that is flat and locally of finite presentation, one defines \cite[VII 3.1.1, p.211]{illusie2} the co-Lie complex $$\ell_G:=\mathrm L\epsilon^*L_{G/S}$$ to be the derived pull-back of the cotangent complex of $G$ over $S$ along the zero section $\epsilon: S\to G$, and the Lie complex  to be $$\ell^\vee_G:=\mathrm R\mathcal Hom (\ell_G,\mathcal O_S).$$ 

Since a group scheme $G$ that is flat over $S$ is always a local complete intersection over $S$ (by \cite[Exp.~VIIB, Cor.~5.5.1]{sga3-1} or \cite[III, \S3, no.~6, p.~346]{DemazureGabriel}), the cotangent complex $L_{G/S}$ has perfect amplitude in $[-1,0]$.  Therefore $\ell_G$ also has perfect amplitude in $[-1,0]$ and $\ell_G^\vee$ has perfect amplitude in $[0,1]$.

If $G$ is smooth over $S$ then $L_{G/S} = \Omega_{G/S}$ is locally free and $\ell_G^\vee$ coincides with the Lie algebra $\mathfrak g$ of $G$.  This is the case in particular if $S$ is the spectrum of a field of characteristic zero, since in that case all group schemes over $S$ are smooth~\cite[\href{https://stacks.math.columbia.edu/tag/047N}{Lem.~047N}]{stacks-project}.

\subsubsection{Small extensions of Artinian rings}
\label{SSS:small}

Infinitesimal deformations of a flat commutative group scheme $G$ that is locally of finite presentation over $S$ are classified in \cite[VII Thm.~4.2.1, p.~239]{illusie2}.  
    The setting for these types of statements is 
a small extension of local Artinian rings with  residue field $k$:
\begin{equation}\label{E:SmExt}
\xymatrix{0\ar[r]& I\ar[r]& R'\ar[r]& R\ar[r]& 0}
\end{equation}
Here $R'\to R$ is a surjective local homomorphism of Artinian local rings with residue field $k$, and $I$ is the kernel of $R'\to R$.  
Recall that if $\mathfrak m$ is the maximal ideal of $R$, then for the extension to be small means that $\mathfrak m \cdot I=0$; this implies that $I$ is naturally a $k$-vector space.    We sometimes restrict, for simplicity, to the situation where $I$ is taken to be a principal ideal; one can always reduce to that case by filtering $I$.  We will typically refer to small extensions by saying  $R' \to R$ is a small extension of an Artinian local ring $R$ with residue field $k$ by a $k$-vector space $I$.

\subsubsection{Illusie's results} \label{S:8}
 
The following theorem is an abbreviated special case of Illusie's main result on deformations of commutative group schemes \cite[VII Thm.~4.2.1, p.~239]{illusie2}:

\begin{teo}[Illusie] \label{T:1}
	Let $R$ be an Artinian local ring with residue field $k$.  Let $R' \to R$ be a small extension of $R$ by a $k$-vector space $I$.  Let $S = \operatorname{Spec} R$.  Let $G$ be a flat, commutative group scheme locally of finite presentation over $R$.  
	\begin{enumerate}[label=(\alph*)]
		\item \label{T:1_i}  There is a natural obstruction in $$\operatorname{Ext}^2_{\mathrm{fpqc}(S)}(G,\ell_G^\vee \mathop\otimes_{\mathcal O_S}^{\mathrm{L}} I)$$ whose vanishing is equivalent to the existence of a flat deformation of $G$ over $R'$.
		\item \label{T:1_ii} If this obstruction vanishes, the set of isomorphism classes of extensions of $G$ forms a torsor under $$\operatorname{Ext}^1_{\mathrm{fpqc}(S)}(G, \ell_G^\vee \mathop\otimes_{\mathcal O_S}^{\mathrm{L}} I).$$
	\end{enumerate}
\end{teo}

Illusie also describes the deformation theory of homomorphisms of commutative group schemes \cite[VII Thm.~4.2.3, p.~240]{illusie2}:
\begin{teo}[Illusie] \label{T:2}
	Let $R$ be an Artinian local ring with residue field $k$.  Let $R' \to R$ be a small extension of $R$ by a $k$-vector space $I$.  Let $S = \operatorname{Spec} R$.  Let $F'$ and $G'$ be flat, commutative group schemes of finite presentation over $R'$.  Let $F$ and $G$, respectively, be their restrictions to $R$.  Suppose that $u : F \to G$ is a homomorphism of group schemes over $R$.
	\begin{enumerate}[label=(\alph*)]
		\item \label{T:2_i}  There is a natural obstruction in $$\operatorname{Ext}^1_{\mathrm{fpqc}(S)}(F, \ell_G^\vee \mathop\otimes_{\mathcal O_S}^{\mathrm{L}} I)$$ whose vanishing is equivalent to the existence of an extension of $u$ over $R'$.
		\item  \label{T:2_ii}  If this obstruction vanishes, the set of extensions of $u$ over $R'$ forms a torsor under $$\operatorname{Ext}^0_{\mathrm{fpqc}(S)}(F, \ell_G^\vee \mathop\otimes_{\mathcal O_S}^{\mathrm{L}} I).$$
	\end{enumerate}
\end{teo}

\begin{rem}\label{R:tan_to_obst} 
In \Cref{T:2}, suppose that $\alpha \in \operatorname{Ext}^1_{\mathrm{fpqc}(S)}(F, \ell_G^\vee \mathop\otimes^{\mathrm{L}} I)$ is the obstruction to lifting $u:F\to G$ to a morphism $u':F'\to G'$, and $F''$ is a second lifting of $F$ to $S'$, that differs from the lifting $F'$ by the element $\beta \in \operatorname{Ext}^1_{\mathrm{fpqc}(S)}(F, \ell_F^\vee \mathop\otimes^{\mathrm{L}} I)$, under the identification  coming from  \Cref{T:1}~\ref{T:1_ii}. Then the obstruction to lifting $u:F\to G$ to a morphism $u'':F''\to G'$ is given by $\alpha-[\beta]$, where $[\beta]$ is the image of $\beta$ under the natural morphism $$ \operatorname{Ext}^1_{\mathrm{fpqc}(S)}(F, \ell_F^\vee \mathop\otimes^{\mathrm{L}} I) \to  \operatorname{Ext}^1_{\mathrm{fpqc}(S)}(F, \ell_G^\vee \mathop\otimes^{\mathrm{L}} I).$$
\end{rem}

\subsubsection{Setting} \label{S:9}

Some care is needed with respect to the categories in which extensions are taken in Theorems~\ref{T:1} and~\ref{T:2}:  

\begin{enumerate}
	\item 
Illusie
works over the fpqc site of schemes over $S=\operatorname{Spec} R$,
and defines $\underline \ell_G$ to be the derived pull-back of
$\ell_G$ to this site \cite[VII (4.1.2.2), p.~230]{illusie2}.  
   We have not distinguished notationally between $\ell_G$ and $\underline\ell_G$ here.

\item 
Illusie views the commutative group scheme
$G$ over $S$, via its functor of points, as a module over the 
constant sheaf of rings $A=\mathbb Z_S$ \cite[\S5.1]{illusie_conf_proc}.   The deformation and
obstruction spaces are then identified with the groups
$\operatorname{Ext}_A^i(G,\underline \ell_G^\vee \otimes^L_{\mathcal
  O_S}I)$,  for $i=0,1,2$, where the extensions are
taken in the derived category of sheaves of $A$-modules over the fpqc
site of schemes over $S$.  We have suppressed $A$ from the notation
but we have included a subscript $\mathrm{fpqc}(S)$ to emphasize that
extensions are taken in the derived  category of fpqc sheaves of abelian groups.

\item 
	As was noted earlier, when $G$ is smooth over $S$, we can identify $\ell_G^\vee$ with the Lie algebra $\mathfrak g$ of $G$, which is a locally free $\mathcal O_S$-module.  The obstruction and deformation spaces arising in Theorems~\ref{T:1} and~\ref{T:2} can therefore be written
		$\operatorname{Ext}^i_{\mathrm{fpqc}(S)}(G, \mathfrak g \otimes_R I)$.
    \end{enumerate}

\subsubsection{Changing sites and identifying tangent-obstruction spaces}\label{S:Site-Change}
The tangent and obstruction spaces in Illusie's theorems can be identified more concretely in the cases we are considering.  In this direction, a primary purpose of this section is to prove \Cref{L:R2gg0}.

\begin{pro} \label{P:4}
	Let $\sigma$ and $\tau$ be the categories of sheaves in two Grothendieck topologies.  Let $\varepsilon^\ast : \tau \to \sigma$ be a left exact
	 functor with right adjoint $\varepsilon_\ast$.\footnote{In other words, $\epsilon : \sigma \to \tau$ is a morphism of topoi \cite[\href{https://stacks.math.columbia.edu/tag/00X9}{Section 00X9}]{stacks-project}.}
 	Suppose that $G$ is a sheaf of abelian groups on $\tau$ and $F$ is a bounded-below complex of abelian groups on $\sigma$ whose cohomology groups are acyclic with respect to $\varepsilon_\ast$.  Then the natural homomorphism
	\begin{equation*}
		\operatorname{Ext}^p(G, \varepsilon_\ast F) \to \operatorname{Ext}^p(\varepsilon^\ast G, F)
	\end{equation*}
	is an isomorphism.
\end{pro}

\begin{proof}
A functor with a right adjoint is right exact, so, when combined with the left exactness of $\epsilon^*$, the existence of a right adjoint of $\varepsilon^\ast$ implies that $\varepsilon^\ast$ is exact.
Since $\varepsilon^*$ is exact,  $\mathrm L\varepsilon^*= \varepsilon^*$, and consequently, $\varepsilon^*$ has a right adjoint $\mathrm R\varepsilon_*$ in the derived category  \cite[\href{https://stacks.math.columbia.edu/tag/09T5}{Lemma 09T5}]{stacks-project}. This gives us
$$
\mathrm R\operatorname{Hom}(\varepsilon^\ast G, F) = \mathrm R \operatorname{Hom}(G, \mathrm R\varepsilon_\ast F).
$$
Applying cohomology to these complexes, we obtain the assertion of \Cref{P:4}, provided we show that $R\varepsilon_\ast F$ is quasi-isomorphic to $\varepsilon_\ast F$.

	For this, choose a quasi-isomorphism $F \xrightarrow\sim I$ where $I$ is a bounded-below complex of injective sheaves of abelian groups on $\mathrm{fpqc}(S)$.  Then $\varepsilon_\ast I^p$ is injective for each $p$ because $\varepsilon_\ast$ has the exact left adjoint $\varepsilon^\ast$ (see  \cite[\href{https://stacks.math.columbia.edu/tag/015Z}{Lemma 015Z}]{stacks-project}). Furthermore $\mathrm R^p \varepsilon_\ast H^q(F) = 0$ for $p > 0$ and all $q$, by assumption.  Therefore the spectral sequence 
	\begin{equation*}
		E^{p,q}_2= \mathrm R^p \varepsilon_\ast H^q (F) \Rightarrow \mathrm R^{p+q} \varepsilon_\ast F
	\end{equation*}
	degenerates at the $E_2$ page to $\varepsilon_\ast H^\bullet(F)$.  
	Therefore the morphism $\varepsilon_\ast F \to \varepsilon_\ast I = \mathrm R \varepsilon_\ast F$ is a quasi-isomorphism, so $\varepsilon_\ast F$ is quasi-isomorphic to $\mathrm R \varepsilon_\ast F = \varepsilon_\ast I$.  
\end{proof}

We may apply this to $\sigma = \mathrm{fppf}(S)$ and $\tau = \mathrm{fpqc}(S)$:

\begin{cor}[Comparing fpqc and fppf Ext groups] \label{C:4}
	Let $G$ be a commutative group scheme over a scheme $S$.  Let $F$ be a bounded below complex of sheaves on $\mathrm{fpqc}(S)$ with quasicoherent cohomology.  Then $\operatorname{Ext}^p_{\mathrm{fpqc}(S)}(G, F) = \operatorname{Ext}^p_{\mathrm{fppf}(S)}(G,F)$. \qed
\end{cor}

We will therefore drop the subscript $\mathrm{fpqc}$ on $\operatorname{Ext}$ in the future.  We may also apply \Cref{P:4} to a closed embedding:

\begin{cor} \label{C:5}
	Let $G$ be a commutative group scheme over a scheme $S$ and let $\varepsilon : S_0 \to S$ be a closed embedding.  Let $F$ be a bounded below complex of sheaves on $\mathrm{fpqc}(S_0)$ with quasicoherent cohomology.  Then $\operatorname{Ext}^p_{S_0}(\varepsilon^\ast G, F) = \operatorname{Ext}^p_S(G, \varepsilon_\ast F)$. \qed
\end{cor}

More specifically:

\begin{cor}[Reduction to Ext groups on the central fiber] \label{L:R2gg0}
	Let $0\to I\to R'\to R\to 0$ be a small extension of an Artinian local ring $(R,\mathfrak m)$ with algebraically closed residue field $k$, and let $G_R$ be a flat, commutative group scheme of finite presentation over $S= \operatorname{Spec}R$.  Let $S_0 = \operatorname{Spec} k$, and write $G_0$ for the restriction of $G$ to $S_0$.
   Then there is a natural isomorphism
$$
	\operatorname{Ext}_S^i(G_R,\ell_{G_R}^\vee\mathop\otimes_RI) \longrightarrow \operatorname{Ext}^i_{S_0}(G_0,\ell_{G_0}^\vee) \mathop\otimes_k I
$$
for all integers $i\geq 0$.
  In particular, if $G_0$ is smooth with Lie algebra $\mathfrak g_0$, then 
$$
	\operatorname{Ext}^i_S(G_R,\mathfrak g_R\mathop\otimes_RI) \longrightarrow \operatorname{Ext}^i_{S_0}(G_0,\mathfrak g_0)\mathop\otimes_k I
$$
is an isomorphism for all integers $i\geq 0$.
\end{cor}

\begin{proof}
	Since $G_0$ is a local complete intersection over $k$ (by \cite[Exp.~VIIB, Cor.~5.5.1, p.~562]{sga3-1} or \cite[III, \S3, no.~6, Thm.~6.1, p.~346]{DemazureGabriel}), the complex $\ell_{G_R}$ is perfect in $[-1,0]$.  As $R$ is local, we can therefore represent $\ell_R$ by a $2$-term complex of free $R$-modules of finite rank.  We identify $\ell_{G_R}$ with such a representation.  Since the $R$-module structure on $I$ is induced from a $k$-vector space structure, we therefore have
	\begin{equation*}
		\ell_{G_R}^\vee \mathop\otimes^{\mathrm L}_R I = \ell_{G_R}^\vee \mathop\otimes_R I = \varepsilon_\ast (\ell_{G_0}^\vee \mathop\otimes_k I),
	\end{equation*}
	where $\varepsilon : S_0 \to S$ is the inclusion.

	By \Cref{C:5}, we therefore have
	\begin{equation*}
		\operatorname{Ext}^i_S(G_R, \ell_{G_R}^\vee \mathop\otimes_R I) = \operatorname{Ext}^i(G_0, \ell_{G_0}^\vee \mathop\otimes_k I ).
	\end{equation*}
	To conclude, we note that $I$ is a finite-dimensional $k$-vector space, so by the additivity of $\operatorname{Ext}^i$, 
	\begin{equation*}
		\operatorname{Ext}^i_{S_0}(G_0, \ell_{G_0}^\vee \mathop\otimes_k I ) = \operatorname{Ext}^i_{S_0}(G_0, \ell_{G_0}^\vee) \mathop\otimes_k I .
	\end{equation*}
\end{proof}

When Theorems~\ref{T:1} and~\ref{T:2} are applied in the setting of \Cref{L:R2gg0}, we can therefore replace $\operatorname{Ext}^i_{\mathrm{fpqc}(S)}(G, \ell_G^\vee \otimes^{\mathrm L} I)$ and $\operatorname{Ext}^i_{\mathrm{fpqc}(S)}(F, \ell_G^\vee \otimes^{\mathrm L} I)$ in the statements by $\operatorname{Ext}^i_{S_0}(G_0, \ell_{G_0}^\vee \otimes I)$ and $\operatorname{Ext}^i_{S_0}(F_0, \ell_{G_0}^\vee \otimes I)$, respectively, where $F_0$ and $G_0$ denote the restrictions of $F$ and $G$ to $S_0$.

\subsection{Proof of \Cref{T:KerMorph}} \label{S:Pf1}

As explained in \Cref{L:R2Artin}, it suffices to prove \Cref{T:KerMorph} in the case where $S=\operatorname{Spec}R$ is the spectrum of an Artinian local ring $(R,\mathfrak m)$ with algebraically closed residue field $k$. 
  Let $f_0:X_0\to Y_0$ be the restriction to the central fiber, with kernel $G_0=\ker(f_0)$, which is assumed to be smooth. We will show that $\ker(f)$ is smooth over $R$.
  
 Inductively, it suffices to consider the following situation.  Let $0\to I\to R'\to R\to 0$ be a small extension of an Artinian local ring $(R,\mathfrak m)$ with algebraically closed residue field $k$, and let $S' = \spec R'$.
 Let $f':X'\to Y'$ be a morphism of abelian schemes over $R'$ with kernel $\ker(f')$, restricting to $f:X\to Y$ over $R$ and to $f_0:X_0\to Y_0$ over the residue field $k$.  Assuming that $\ker(f)$ is smooth over $R$, we must show that $\ker(f')$ is smooth over $R'$.  Note that since we are working over an Artinian local ring, it suffices to show that $\ker(f')$ is flat, since the only geometric fiber $\ker(f_0)$ is assumed to be smooth.  
 
For brevity, set $G'=\ker(f')$, $G=\ker (f)$, and $G_0=\ker(f_0)$.  We argue first that $G$ extends to a smooth, proper group scheme over $S'$. 
Note that by \Cref{T:1}~\ref{T:1_i} and \Cref{L:R2gg0}, obstructions to extending $G$ to a flat commutative group scheme $G''$ over $S'$ lie in $\operatorname{Ext}^2_k(G_0, \mathbb G_a) \otimes_k \mathfrak g \otimes_k I$.
 This group vanishes since $G_0$ is a smooth proper group scheme over a field (see \Cref{P:13}), however, we prefer to give a direct proof here of the existence of the extension.  
 Since $G$ is smooth and proper, it is an extension of a finite \'etale group scheme $B$ over $S$ by an abelian scheme $A$ over $S$.  The deformations of $B$ are unobstructed because $B$ is \'etale  ($B$ is a disjoint union of copies of the base) and the deformations of $A$ are unobstructed because $A$ is an abelian scheme \cite[Thm.~(2.2.1), p.~273]{oortFGS}.  Therefore $A$ lifts to an abelian scheme $A'$ over $S'$ and $B$ lifts to a finite \'etale group scheme $B'$ over $S'$.  To complete the lifting of $G$, we only need to lift the class of the extension.  But $B$ has order invertible in $k$ (because $\Hom(G_0, \mathbb G_a) = 0$).  There is therefore a positive integer $N$ such that $NB = 0$ and $N$ is invertible in $k$.  We have an exact sequence of sheaves on the big \'etale site of $S'$:
\begin{equation*}
	0 \to \underline{\hom}(B',A') \to \underline{\ext}^1(B',A'[N]) \to \underline{\ext}^1(B',A') \to 0
\end{equation*}
But both $\underline{\hom}(B',A') = \underline{\hom}(B',A'[N])$ and $\underline{\ext}^1(B',A'[N])$ are representable by algebraic spaces that are \'etale over $S'$ since both $B'$ and $A'[N]$ are \'etale over $S'$.  Therefore $\underline{\ext}^1(B',A')$ is \'etale over $S'$ and the class of $G$ in $\ext^1(B,A)$ lifts (uniquely) to an element $\ext^1(B',A')$.  We write $G''$ for the corresponding lift of $G$ to $S'$.

 	Our next step is to modify $G''$ to ensure that the $S$-homomorphism $G \to X$ extends to an $S'$-homomorphism $G'' \to X'$.  For this, we observe that the exact sequence
	\begin{equation*}
\xymatrix{		0 \ar[r]& \mathfrak g \ar[r] &\mathfrak x \ar[r]& \mathfrak y \ar[r]& \mathfrak h \ar[r]& 0}
	\end{equation*}
	is split (it is a sequence of $k$-vector spaces).
	 Therefore, the sequence of $k$-vector spaces 
	\begin{equation}\label{E:Nat_Ext_ES}
		\xymatrix{0 \ar[r]& \operatorname{Ext}^1_k(G_0, \mathfrak g) \mathop\otimes_k I
		\ar[r]& \operatorname{Ext}^1_k(G_0, \mathfrak x) \mathop\otimes_k I
		\ar[r]& \operatorname{Ext}^1_k(G_0, \mathfrak y) \mathop\otimes_k I}
	\end{equation}
	is also exact.
	 
	Since there is a flat, commutative extension $G''$ of $G$ to $S'$, \Cref{T:1}~\ref{T:1_ii} implies that the choices of $G''$ are parameterized by $\operatorname{Ext}^1_k(G_0, \mathfrak g) \mathop\otimes_k I$.  For a fixed choice of $G''$, the obstruction to deforming the map 
	$G \to X$ to a map $G'' \to X'$ lies in $\operatorname{Ext}^1_k(G_0, \mathfrak x) \otimes_k I$ (see \Cref{T:2}~\ref{T:2_i}). The 
	image of this  extension in the group on the right in \eqref{E:Nat_Ext_ES}, namely $\operatorname{Ext}^1_k(G_0, \mathfrak y) \otimes_k I$,
	is the obstruction to deforming the composition $G \to Y$ to $G'' \to Y'$; this latter obstruction is zero, since the morphism $G\to Y$ is the zero morphism, and can therefore be trivially lifted to the zero morphism  $G''\to Y'$.  Therefore, by the exactness of~\eqref{E:Nat_Ext_ES}, the obstruction to deforming the map $G \to X$ actually lies in the image of 
	$\operatorname{Ext}^1_k(G_0, \mathfrak g) \otimes_k I$.
	But  since $\operatorname{Ext}^1_k(G_0, \mathfrak g) \otimes_k I$  acts simply transitively on the choices of $G''$ deforming $G$, 
	and the corresponding effect of this action of $\operatorname{Ext}^1_k(G_0, \mathfrak g) \otimes_k I$ on the  obstruction to deforming $G\to X$ is given via the natural map in \eqref{E:Nat_Ext_ES} (see \Cref{R:tan_to_obst}), 
	it follows that there is some choice of deformation $G''$ of $G$ such that the homomorphism $G \to X$ extends to $G'' \to X'$.  Replacing $G''$ with this choice, we have a homomorphism $u : G'' \to X'$.

	By composition, we obtain a morphism $G'' \to X' \to Y'$ extending the zero morphism $G \to Y$. We also have the zero morphism $G''\to Y'$ extending the zero morphism $G\to Y$.  By \Cref{T:2}~\ref{T:2_ii}, the difference between  these two morphisms lies in $\operatorname{Ext}^0_k(G_0,\mathfrak y)$, which is equal to zero by assumption.  Thus the composition $G''\to X'\to Y'$ is the zero morphism, so that $G''\subseteq G'=\ker(f')$.

Finally, since $G''$ is flat, it follows from Nakayama's lemma that $G''=G'$. 
Indeed, first note that since $G$ and $X$ are assumed to be smooth, and $G\hookrightarrow X$ is a closed embedding, we have that $G''\to X'$ is a closed embedding (e.g., \cite[Rem.~3.4.10]{sernesi}).  Now 
let $J$ be the ideal of $G''$ in $G'$. We have a short exact sequence 
$$
	\xymatrix{0\ar[r]&  J\ar[r]& \mathcal O_{G'} \ar[r]& \mathcal O_{G''} \ar[r]& 0.}
$$
	Since $\mathcal O_{G''}$ is assumed to be flat over $S'$, applying $(-) \otimes_{\mathcal O_{S'}} k$ preserves the exactness of the above sequence.  Therefore the ideal of $\mathcal O_{G_0}= \mathcal O_{G''}\otimes_{\mathcal O_{S'}} k$ in $\mathcal O_{G_0} = \mathcal O_{G'} \otimes_{\mathcal O_{S'}} k$ is $J \otimes_{\mathcal O_{S'}} k$.   Hence $J \otimes_{\mathcal O_{S'}} k = 0$, so $J = 0$ by Nakayama's lemma.

	Thus $G'' = G'$ and $G'$ is therefore flat, hence smooth over $S'$. \qed

\section{Examples where the image is not an abelian scheme}\label{S:Examples}

In this section we give two examples (\Cref{E:Serre} and \Cref{E:charp}) of homomorphisms $f:X\to Y$ of abelian schemes over DVRs where the image is not an abelian scheme.   In both cases, the image is a flat proper connected group scheme.  Restricting those examples to Artinian rings over the closed point, one can obtain examples of morphisms over Artinian rings where the image is not an abelian scheme; in those cases, the image is a proper connected group scheme, but is not flat (the image of the restriction, which is not flat, is a closed subscheme of the restriction of the image, which is flat; see \Cref{R:ImNotStable} for more details).   This section can be viewed as an expansion of \cite[Exa.~8, \S 7.5, p.190]{BLR}, which is due to Serre.

\subsection{The main strategy}

The basic observation is the following:
 
\begin{pro}[Serre]\label{P:Serre}
Let  $G$ be a finite flat   group scheme over an integral scheme $S$. Assume  there exist abelian varieties $X$ and $X'$ over $S$ such that $G$ admits a closed embedding of group schemes  $G\hookrightarrow X'$,  and homomorphism $u:G\to X$ that is not a closed embedding  of $S$-group schemes, but is a closed embedding of group schemes over the generic point $\eta$ of $S$.  Then there is a homomorphism of $S$-group schemes $u':X'\to Y'$ so that the image $u'(X')$ is not an abelian subscheme of $Y'$.  
\end{pro}

\begin{proof} We follow the idea from  \cite[Exa.~8, \S 7.5, p.190]{BLR}.
Consider the push-out diagram of commutative group
  schemes
  \begin{equation}
    \label{D:pushout}
\xymatrix{
G \ar[r]^<>(0.5)u \ar@{^(->}[d]& X \ar[d]\\
X' \ar[r]^{u'}& Y',
}
\end{equation}
where  $Y'$ is the quotient of $X'\times_S X$ by
the action of $G$.  
Note that $G$ is a finite flat sub-group scheme of $X'\times_S X$ via the product of the inclusion $G\subseteq X'$ and $u$.   
Since $Y'$ is the quotient
of an abelian $S$-scheme by a finite flat sub $S$-group scheme, it is an
	abelian $S$-scheme.
      
Although $u'_\eta$ is a closed embedding, 
\eqref{D:pushout} shows that, since $u$ is not a closed embedding,
$u'$ is similarly not a closed embedding.
      We can conclude that the image of
$u'$ is not an abelian subscheme of $Y'$.  

Indeed, suppose the image
of $u'$ were an abelian subscheme $X''\subseteq Y'$.  Then we
would have a surjection of abelian schemes $X'\twoheadrightarrow X''$;
	the kernel would be flat~\cite[Lem.~6.12, p.~122]{GIT}, but being trivial on the generic fiber, the
kernel would be trivial everywhere, and therefore $X'\to X''$ would be
an isomorphism.  This contradicts the hypothesis that $u'$ is not a closed
embedding. 
\end{proof}

With this proposition, we simply need to find finite flat group schemes satisfying the conditions of the proposition.  Clearly, from \Cref{T:ImIsAbSch}, one cannot find such examples in characteristic $0$.  In the next sections, we explain how to find such examples in mixed and positive characteristic.

\begin{rem}\label{R:alt-PSerre-view}
An alternative way to interpret \Cref{P:Serre} is to consider the closed embedding $X'=X'\times_S \{1\}\subseteq X'\times_S X$, and the diagonal embedding $G\subseteq X'\times_SX$, and then consider the intersection $X'\cap G:=(X'\times_S\{1\})\times_S G$, which gives us a closed sub-group scheme of $X'$.    Then by assumption, $(X'\cap G)_\eta$ is the trivial group, while there exist points of $s$ of $S$ such that $(X'\cap G)_s$ is non-trivial.  The morphism $u':X'\to X'\times_S X \to Y'$ of \eqref{D:pushout} has kernel $(X'\cap G)$, and one can interpret \Cref{P:Serre} as saying that the quotient of $X'$ by the kernel $(X'\cap G)$ of $u'$ is not an abelian scheme.   
\end{rem}

As our examples will be constructed over DVRs, we make several notes:

\begin{rem}[Images of abelian schemes over a DVR are flat] \label{R:FlatDVR}
	If $R$ is a DVR, then the image of a morphism of abelian schemes over $R$ is flat proper group scheme.  Indeed, if $f : X \to Y$ is a morphism of abelian schemes, then $f(X)$ is irreducible since $X$ is, is reduced since $X$ is (by the definition of the scheme theoretic image), and dominates $\operatorname{Spec}R$ since $X$ does. Being reduced, $f(X)$ has no embedded points, and so every associated point of $f(X)$ dominates $\operatorname{Spec}R$, implying that $f(X)$ is flat over $R$.
	 	                 \end{rem}

\begin{rem}[Conditions for the image to be an abelian scheme over a DVR]\label{R:DVRconds}
In light of \Cref{R:FlatDVR}, given any homomorphism $f:X\to Y$ of abelian schemes over the spectrum $S$ of a DVR, then the image $f(X)$ is an abelian scheme if and only if the fiber $f(X)_s$ over the special point $s$ of $S$ is smooth.  Indeed, as the image is a flat proper group scheme, it will fail to be an abelian scheme if and only if it is not smooth over $S$, which can be checked on the fibers.   Since localization is flat, and formation of images commutes with flat base change, we have that $f_\eta(X_\eta)=f(X)_\eta$ is an abelian variety, where $\eta$ is the generic point of $S$ (since images of abelian varieties are abelian varieties).  
Therefore, it must be that $f(X)_s$ is not smooth.  As another variation, in  \cite[\S4]{FS08}, it is also shown using Zariski's main theorem that $f(X)$ is an abelian scheme if and only if $f(X)$ is normal.  
\end{rem}

\begin{rem}[Failure of base change for images of abelian schemes over a DVR]\label{R:BC}
\Cref{T:ImIsAbSch} implies that if the image of a morphism of abelian schemes $f:X\to Y$ is an abelian scheme, then the image is stable under base change.  However,  if the image of a morphism of abelian schemes is \emph{not} an abelian scheme, then the image need not be stable under base change.  
Indeed, in light of the previous remarks, given any homomorphism $f:X\to Y$ of abelian schemes over the spectrum $S$ of a DVR, if the image $f(X)$ is not an abelian scheme, then base change for the image fails over the special point $s$ of $S$.
As we saw in \Cref{R:DVRconds}, $f_\eta(X_\eta)=f(X)_\eta$ is an abelian variety, where $\eta$ is the generic point of $S$, and  $f(X)_s$ is not smooth.  Considering the natural inclusion $f_s(X_s)\subseteq f(X)_s$, we have that the former is an abelian scheme  and therefore smooth, so that the containment  is not an equality.  
\end{rem}

\subsection{Reminders on finite group schemes and abelian varieties}
\label{SS:fingroupschemes}

For use in examples, we establish some notation concerning certain
finite flat group schemes.  (See, e.g., \cite[\S 1.2]{oortCGS} or  \cite[\S 2]{shatz-gpschemes} for more details of the facts reviewed here.) For our purposes, it suffices to work over 
an affine scheme $S = \spec R$.  Fix a prime $p$.  
\medskip 
\begin{description}

  \item [\'etale] We have the constant group scheme
    $\underline{\integ/p\integ}_R$, with underlying scheme
\[(\underline{\integ/p\mathbb Z})_R =\spec  \underbrace{R\oplus \cdots
    \oplus  R}_{\text{$p$ times}}\]
and multiplication law given by addition in $\integ/p\integ$.

\item[multiplicative] We have the multiplicative group scheme
  $(\mmu_p)_R = \ker [p]: \gp_{m,R} \to \gp_{m,R}$, with underlying scheme
  \[
    (\mmu_p)_R = \spec R[y]/(y^p-1)
  \]
  and comultiplication law $y\mapsto y_1\otimes y_2$.

  \item[additive] If $pR = (0)$, we have the additive group scheme
    which is the kernel of the Frobenius morphism
    $(\aalpha_p)_R = \ker F: \gp_{a,R} \to \gp_{a,R}^{(p)}$, with
    underlying scheme
    \[
      (\aalpha_p)_R = \spec R[z]/(z^p),
    \]
    and comultiplication law $z \mapsto z_1 \otimes 1 + 1 \otimes
    z_2$.
  \end{description}

  \begin{rem}
    For context, if 
$R = k$ is the spectrum of an algebraically closed field of
characteristic $p$, then any group scheme of rank $p$ is isomorphic to
exactly one of those schemes, and there are no nontrivial morphisms
between them; moreover, $(\underline{\integ/p\integ})_k$ is \'etale,
while both $\mmu_{p,k}$ and $\aalpha_{p,k}$ are connected and
nonreduced.  Cartier duality exchanges $\underline{\integ/p\integ}_k$
and $\mmu_{p,k}$, while $\aalpha_{p,k}$ is self-dual.  The
endomorphisms of $\underline{\integ/p\integ}_k$ and $\mmu_{p,k}$ are
discrete; both are isomorphic to $\integ/p\mathbb Z$, while
$\operatorname{End}(\aalpha_{p,k}) \iso k$. We refer the reader to \cite[\S 14]{mumfordAV} for more details. 
 In contrast, if $p$ is invertible in $R$, then $\mmu_{p,R}$ and
$\underline{\integ/p\mathbb Z}_R$ are both \'etale (for the former, use the Jacobian criterion, while the latter is \'etale by construction), and are
\'etale-locally isomorphic (more details below).
\end{rem}

We now  review  morphisms $\underline{\mathbb Z/p\mathbb Z}_R\to \mmu_{p,R}$ in more detail, as the existence of a certain type of such morphism is crucial to \Cref{E:Serre}.  Certainly this is all well known, but we find it is helpful to include these details in explaining  \Cref{E:Serre}.  
Given any ring $R$, a homomorphism of $R$-group schemes  $\underline{\integ/p\integ}_R \to 
  \mmu_{p,R}$  is induced by   a homomorphism $ R[y]/(y^p-1) \to  R\oplus \cdots \oplus  R$ of $R$-algebras.  Such a homomorphism is governed by the image of $y$ in each component, each of which must be a $p$-th root of unity from the ring homomorphism property.   In particular, if $\zeta_p\in R$ is a $p$-th root of unity, then we obtain a homomorphism of $R$-group schemes 
\begin{equation}
  \label{D:g}
\xymatrix{g= g_{\zeta_p}: \underline{\integ/p\integ}_R \ar[r]&
  \mmu_{p,R}}
\end{equation}
corresponding to the ring
homomorphism
\begin{equation}
  \label{D:phi}
  \xymatrix@R=.2em{
 R[y]/(y^p-1) \ar[r]^{\phi = \phi_{\zeta_p}}& R\oplus \cdots \oplus  R \\
y\ar@{|->}[r] &  (1,\zeta_p,\dots,\zeta_p^{p-1}).
}
\end{equation}
In fact, one can see from the homomorphism property for the group schemes that any morphism $\underline{\mathbb Z/p\mathbb Z}_R\to \mmu_{p,R}$ can be described in this way.

If $p \in R^\times$ (the group of units)  and $\zeta_p\in R$
is a primitive $p$-{th} root of unity, then $g$  \eqref{D:g} is an isomorphism.  
Indeed, it suffices to check that $g$ is an isomorphism after a faithfully flat base change; arguing as in \Cref{L:R2Artin}, it suffices to consider the case where $R$ is an Artinian local ring where $p$ is invertible.  Since both $\underline{\mathbb Z/p\mathbb Z}_R$ and $\mmu_{p,R}$ are \'etale over $R$, they consist of disjoint unions of components, and $g$ is the identity on components (although components can be identified under $g$).  Thus we only need to show that the morphism is a bijection on components, which we can check on the special fiber.  So we have reduced to showing that $g$ is an isomorphism when $R$ is an algebraically closed field of characteristic not equal to $p$.  In this case, the Chinese Remainder Theorem gives that  \eqref{D:phi} is an isomorphism.
Note that conversely, if $R$ is a field of characteristic $p$, then  the same argument shows that the only morphism  $\underline{\mathbb Z/p\mathbb Z}_R\to \mmu_{p,R}$ is the constant morphism.

\begin{rem}\label{R:e-ge-p-1}
If $R$ is a DVR of mixed characteristic, with valuation $\nu$, then
the condition that $R$ contains all $p$-th roots of unity implies that
$e:=\nu(p)$ is a positive multiple of $p-1$, and thus $e\ge p-1$  (\cite[Ch.~IV, \S 4, Prop.~17]{serre4}).  In particular, if there is a morphism $\underline{\integ/p\integ}_R \to \mmu_{p,R}$ that is generically an isomorphism, then $e \ge p-1$.
 \end{rem}

\begin{rem}[Flatness of the image and kernel of $g$]
	The image of $g$ is a proper group scheme of $(\mmu_p)_R$ over $R$, which can be  flat.  For instance, if $R$ is a field of  characteristic not equal to $p$, then the image is flat.  More interestingly, if $R$ is a DVR, then the image of $g$ is flat (it is torsion-free since $\mathcal O_{\underline{\integ/p\integ}}$ is flat, hence torsion-free).  
  
	Similarly, the kernel of $g$ is a proper subgroup scheme of $(\mathbb Z/p\mathbb Z)_R$, which also can be flat.  For instance, if  $R$ is a field of  characteristic not equal to $p$, then the kernel is flat.   

	It is not hard to find examples of rings over which neither the image nor kernel of $g$ is flat.  If $R$ is any commutative ring containing an element $\zeta$ such that $\zeta^p = 1$ then we obtain a morphism $g : \underline{\integ/p\integ}_R \to \mmu_{p,R}$ sending the section $1$ to $\zeta$.  The intersection of the kernel of $g$ with the section $1$ is isomorphic to $\operatorname{Spec}(R/(1-\zeta))$, which can fail to be flat over $R$ (if $R = \mathbb Z[\zeta] / (\zeta^p - 1)$, for example).

	The image is the spectrum of an $R$-subalgebra of $\Gamma(\underline{\integ/p\integ}, \mathcal O) = R^p$.  If $R$ is a DVR, it will therefore be torsion free, hence flat.  We will explore a situation where the image is not flat in \Cref{Ex:imagenotflat}.
	 \end{rem}

\begin{exa}[Example of non-flat kernel and image for $g$]  \label{Ex:imagenotflat}
	Let $R = \integ/4\integ$ and consider the map
	\begin{equation*}
		g : \underline{\integ/2\integ}_R \to \mmu_{2,R}
	\end{equation*}
	that sends the nontrivial section of $\underline{\integ/2\integ}_R$ to $-1$.  This corresponds dually to the ring homomorphism
	\begin{equation*}
		\phi : R[y] / (y^2 - 1) \to R \times R : y \mapsto (1,-1) .
	\end{equation*}
	The schematic image of $g$ is the spectrum of the image of $\phi$.  But $\phi(y-1)$ is not a multiple of $2$ while $\phi(2(y-1)) = 0$.  Thus $2(y-1)$ is a nonzero element of $2R \otimes_R \operatorname{im}(\phi)$ whose image in $\operatorname{im}(\phi)$ is zero, and so $\operatorname{im}(\phi)$ is not flat over $R$.n

In a similar fashion, for any prime $p$, any $r \ge 2$, and any Artinian
local ring $R$ admitting an inclusion $\integ[\zeta_p]/(p^r)
\hookrightarrow R$, there exists a morphism of finite flat commutative
group schemes $(\underline{\integ/p\mathbb Z})_R \to (\mmu_p)_R$ with neither kernel nor image
 flat over $\spec R$. 
 \end{exa}

\subsubsection{Canonical lifting}

We provide a quick sketch of the construction of the Serre--Tate
canonical lift, although all we need is \Cref{E:canonical}.

Let $A_0/\kappa$ be an ordinary $g$-dimensional abelian variety over a
perfect field $\kappa$. Then there is a canonical isomorphism of $p$-divisible
groups
\[
  A_0[p^\infty] \iso (\widehat A_0[p^\infty]^\et)^t \times A_0[p^\infty]^\et,
\]
where $A_0[p^\infty]^\et$ is the maximal \'etale quotient of
$A_0[p^\infty]$, thus an \'etale $p$-divisible group of height $g$;
$\widehat A_0$ is the dual abelian variety; and $^t$ denotes Serre dual.  In
particular, $A_0[p^\infty]$ is canonically isomorphic to the product
of a canonically-defined \'etale $p$-divisible group and the Serre dual of a
second canonical \'etale $p$-divisible group.

Now let $R$ be a complete Noetherian local ring with residue field
$\kappa$.  Then $A_0[p^\infty]^\et$ and $\widehat A_0[p^\infty]^\et$ both
canonically deform to \'etale $p$-divisible groups over $G$ and $\widehat G$
over $R$, and the Serre--Tate canonical lifting of $A$ to $R$ is the
deformation with $p$-divisible group $A[p^\infty] \iso (\widehat G)^t \times
G$.

\begin{exa}
  \label{E:canonical}
  Let $\kappa$ be a perfect field of characteristic $p>0$, let $R$ be a complete Noetherian local
ring with residue field $\kappa$, and let $E_0/\kappa$ be an elliptic curve with
$E_0[p](\kappa) \iso \integ/p\mathbb Z$; such an elliptic curve exists.
(This is well-known, but for convenience, we remind the reader: By base change, it suffices to show existence over $\mathbb F_p$.  For this, given an elliptic curve $E/\mathbb F_p$ with trace of Frobenius,  $a$, we have 
$|E(\mathbb{F}_p)| = p+1-a$.  If one can find $E/\mathbb F_p$ with trace of Frobenius $a=1$, then by counting points, one has  $E(\mathbb{F}_p)= E[p](\mathbb F_p)$, and consequently, $E[p](\mathbb{F}_p)=\mathbb Z/p\mathbb Z$.  For a prime $p$, any number $a$ satisfying $|a|\le 2\sqrt p$ arises as the trace of Frobenius for some elliptic curve.)
     Then $E_0$ is ordinary; the
condition on $\kappa$-points forces an identification of group schemes  $E_0[p] \iso \mmu_{p,\kappa}
\times_\kappa \underline{\integ/p\integ}_\kappa$, and its canonical lift $E$ of $E_0$ to $R$
satisfies
\begin{equation}
  \label{E:cansplit}
  E[p] \iso \mmu_{p,R} \times_R \underline{\integ/p\integ}_R.
\end{equation}
                  \end{exa}

\subsection{Examples of homomorphisms of abelian schemes where the image is not an abelian scheme}

\begin{exa}[{Serre's example over a DVR of mixed characteristic
    \cite[Exa.~8, \S 7.5, p.190]{BLR}}] \label{E:Serre} 
    Here we give an example of a morphism of abelian schemes over a
    DVR of mixed characteristic  that has image a flat proper group
    scheme that is not an abelian scheme.

Let $R$ be a DVR containing all $p$-th
roots of unity, with perfect residue field $k$ of characteristic
$p>0$.  Note that the ramification index of $R$ satisfies $e := v(p)
\ge p-1$ (\Cref{R:e-ge-p-1}), and so this situation is \emph{not}
automatically covered by \Cref{T:ImIsAbSch}.
Now let $E$ and $E'$ be elliptic curves over $R$ (i.e., one-dimesional
abelian schemes) admitting respective closed embeddings
\begin{equation} \label{D:torsioncanonical}  
  \xymatrix@R=1em{
   \mmu_{p,R} \ar@{^(->}[r] & E \\
    \underline{\integ/p\integ}_R \ar@{^(->}[r] & E'.
  }
\end{equation}
In fact, one could take $E = E'$ constructed as follows.  Let $E_0/k$ be an ordinary elliptic curve with $E_0[p](k) \iso
(\integ/p\mathbb Z)$; this can already be accomplished over $\ff_p$.  Let $E$ be its canonical lift to $R$ (\Cref{E:canonical}); the
existence of both embeddings above comes from the
splitting \eqref{E:cansplit}.

The key point is that the hypothesis on $R$ lets us construct the  morphism $g:
\underline{\integ/p\integ}_R \to \mmu_{p,R}$ in \eqref{D:g} with the property that $g$ is an isomorphism on the generic fiber, but not an isomorphism on the special fiber. 
                                          Now let $u$ be the
  composition of $g$  with the closed embedding \eqref{D:torsioncanonical}:
$$
  \xymatrix{u:(\underline{\mathbb Z/p\mathbb Z})_R\ar[r]^<>(0.5){g}&
    (\mmu_p)_R\ar@{^(->}[r] & E.}
$$ 
Note that $u$ is a closed
  embedding on the generic fiber but, by considering the morphism over
  the special point, we see that $u$ is not a closed
  embedding. 
  With this, we are done from \Cref{P:Serre}. 
  
 For clarity, consider the push-out diagram of commutative group
  schemes
  \begin{equation}
    \label{D:pushout2}
\xymatrix{
(\underline{\mathbb Z/p\mathbb Z})_R \ar[r]^<>(0.5)u \ar@{^(->}[d]& E \ar[d]\\
E' \ar[r]^{u'}& F',
}
\end{equation}
where  $F'$ is the quotient of $E'\times_R E$ by
the action of $(\underline{\mathbb Z/p\mathbb Z})_R$.   While the
image of $u'$ is a proper flat group scheme (\Cref{R:FlatDVR}), it is
\emph{not} an abelian scheme (as explained in the proof of \Cref{P:Serre}).

                \end{exa}

\begin{exa}[Fakhruddin--Srinivas' example over a DVR of pure characteristic $p>0$]\label{E:charp}                 
    Here we
  give a similar example of a morphism of abelian schemes over a DVR of pure characteristic $p > 0$ whose image is a flat proper group scheme that is not an abelian scheme.
  We follow the hint in \cite[p.~132]{FS08}.
  We refer the reader to \cite[\S1]{lioort} for
	reminders on supersingular abelian varieties,
and start by recalling the Moret-Bailly construction \cite[\S{}A.9]{lioort}.
  Let $k$ be a field of characteristic $p>0$, let $E_0/k$ be a
  supersingular elliptic curve, and let $A_0$ be the superspecial
	abelian surface $A_0 = E_0 \times_k E_0$.  Since $E_0$ is supersingular, it admits an inclusion $h : \aalpha_p \to E_0$ (in fact, a unique one, up to isomorphism, since the $p$-torsion of $E_0$ is a non-split extension of $\aalpha_p$ by $\aalpha_p$).  Since the endomorphism ring scheme of $\aalpha_p$  is $\mathbb G_a$, there is a $2$-dimensional family of maps $(ah, bh) : \aalpha_p \to A_0$, depending on parameters $a,b \in \mathbb G_a$ (in fact, these are all homomorphisms $\aalpha_p \to A_0$).  Then $$P := \proj\hom(\aalpha_p,
A_0) \iso \proj^1_k$$ parametrizes sub-group schemes of $A_0$ which are
isomorphic to $\aalpha_p$.  

  Let $E = E_0\times_k P$ and $A = A_0 \times_k
P$; each is an abelian scheme over $P$.
The graph of the natural
morphism $\aalpha_p \times_k \hom(\aalpha_p, A_0) \to A_0 \times_k
\hom(\aalpha_p, A_0)$ descends to give a sub-group scheme $G\subseteq A$.  Note that  $G \iso \aalpha_p \times P$ is a finite flat group
scheme over $P$.  Let $\iota:G \hookrightarrow A$ be the inclusion.

Let $E_1 = (E_0 \times_k \st 1)\times_k P$ and $E_2 = (\st 1 \times_k
E_0)\times P$, with inclusions $\iota_i: E \stackrel{\sim}{\to} E_i\hookrightarrow A$ and
projections $\pi_i: A \to E_i$.  Let $H_i = G\times_A E_i$.  There is
a unique closed point $t_i \in P$ such that, for $t\in P$ with residue
field $\kappa$, 
\[
  H_{i,t} \iso \begin{cases}\st{1} & t\not = t_i\\
    \aalpha_{p,\kappa} & t=t_i;
  \end{cases}
\]
and $t_1\not = t_2$.
The composition $G\hookrightarrow A \stackrel{\pi_i}{\to}E_i$ is a closed
embedding away from $s_i := t_{2-i}$, and fails to be an embedding at
$s_i$.
So, let $S = P-\{s_2\}\cong \mathbb A^1_k$, and pull back
these maps to $S$.  Then
\[
  \xymatrix{G_S \iso \alpha_{p,S} \ar@{^(->}[r] & A_S
    \ar@{->>}[r]^<>(0.5){\pi_1} & E_{1,S}
  }
\]
is a closed immersion at the generic point of $S$, but fails to be an
immersion at $s_1$, while the composition 
\[
  \xymatrix{G_S  \ar@{^(->}[r] & A_S
    \ar@{->>}[r]^<>(0.5){\pi_2} & E_{2,S}
  }
\]
is a closed immersion.  The construction of \Cref{P:Serre} now
produces an example of a morphism of abelian schemes over $S$  whose image is not
an abelian scheme.
Replacing $S$ with the spectrum 
$\spec \mathcal O_{P,s_1}\hookrightarrow P$ gives an example over a DVR of pure characteristic $p$.
\end{exa}

\begin{exa}[Continuation of \Cref{E:charp}]
  \label{E:charpbis}
We now view  \Cref{E:charp} from the perspective of \Cref{R:alt-PSerre-view}. 
    Note that $E_0$ has a unique sub-group scheme isomorphic to
$\aalpha_p$, and $E_0/\aalpha_p \iso E_0^{(p)}$, the Frobenius
twist of $E_0$.  Let us further assume that $k = \ff_{p^2}$, and that
$E_0 \not \iso E_0^{(p)}$.  (This is possible for most primes $p$.  In
particular, such an $E_0$ exists if $p \ge 71$; see e.g.,   \cite[\S 1]{sankaran-ss} and  \cite{ogg75}.) 
 Let $u'$ be the composition
\[
  \xymatrix{
u': E \ar@{^(->}[r]^{\iota_1} & A \ar@{->>}[r] & A/G
  }
  \]
  from \Cref{P:Serre} (see \Cref{R:alt-PSerre-view}).
 We claim  that the image of $u'$ is not an abelian scheme. 
  On one hand, at the generic point we have $u'(E)_\eta  =
E_\eta/H_{1,\eta} \iso E_\eta \iso
E_0 \times \eta$.  Consequently, the only abelian scheme over $P$ with
generic fiber $u'(E)_\eta$ is $E$ itself.
On the other hand, essentially because set-theoretic image commutes with base
change and because $u'$ is proper,
   we have an inclusion of schemes $u'(E_{t_1})  = E_{t_1}/H_{1,t_1}\subseteq
u'(E)_{t_1}$ which induces an isomorphism of topological spaces
\cite[pp.~216 and~218]{EH00}.  In
particular, we have
\[
  (u'(E)_{t_1})_{{\mathrm{red}}} \iso E_0^{(p)}\not\iso E_0,
\]
and so $u'(E)$ is not an abelian scheme.
Moreover, since $u'(E)$ is flat (\Cref{R:FlatDVR}; flatness is local), the
special fiber $u'(E)_{t_1}$ is nonreduced; otherwise, $u'(E)$ would
be a proper, flat connected smooth group scheme, thus an abelian
scheme.
 \end{exa}

\section{Deformations of smooth proper group schemes} \label{S:4}

The purpose of this section is to establish \Cref{P:2}, that deformations of smooth proper commutative group schemes are unobstructed.  We then discuss the relationship between this result and some other results in the literature.

\subsection{Generalities on extensions of group schemes} \label{S:3}

Bertolin and Tatar describe a canonical partial resolution of any
abelian group $G$ (with no geometric structure) \cite{BertolinTatar}:
\[
\xymatrix{L(G)_5 \ar[r]& L(G)_4 \ar[r]& L(G)_3 \ar[r]& L(G)_2 \ar[r]& L(G)_1 \ar[r]& G \ar[r]& 0}
\]

each $L(G)_i$ is a direct sum of sheaves $\mathbb Z[G^r]$.\footnote{This notation refers to the free abelian group generated by $G^r$.}
This complex is closely related to a complex described by Breen \cite{Breen1969b} that computes the homology of the Eilenberg--MacLane spectrum $HG$.  The complexes agree in low degrees but begin to diverge at $L(G)_4$.  Bertolin and Tatar's sequence will be the more convenient one for us because it is exact up to $L(G)_4$, which will be enough to compute $\operatorname{Ext}^2(G, -)$.

The following are the low-degree terms:
\begin{equation*} \begin{aligned}
	L(G)_1 & = \mathbb Z[G] \\
	L(G)_2 & = \mathbb Z[G^2] \\
	L(G)_3 & = \mathbb Z[G^3] \oplus \mathbb Z[G^2] \\
	L(G)_4 & = \mathbb Z[G^4] \oplus \mathbb Z[G^3] \oplus \mathbb Z[G^3] \oplus \mathbb Z[G^2] \oplus \mathbb Z[G]
\end{aligned}
\end{equation*}
Following Bertolin--Tatar, we write $[x|y|z]$ and $[x|_2y]$ for basis elements of $L(G)_3$ and $[x|y|z|w]$, $[x|y|_2z]$, $[x|_2y|z]$, and $[x|_3y]$ for basis elements of $L(G)_4$.  In terms of these basis elements, the differentials are given by the following formulas:
\begin{equation*} \begin{aligned}
	\partial_2 [x|y] & = [x + y] - [x] - [y] \\
	\partial_3 [x|y|z] & = [x + y|z] + [x|y] - [x|y+z] - [y|z]  \\
	\partial_3 [x|_2y] & = [x|y] - [y|x] \\
	\partial_4 [x|y|z|w] & = [y|z|w] - [x+y|z|w] + [x|y+z|w] - [x|y|z+w] + [x|y|z]  \\
	\partial_4 [x|y|_2|z] & = [x|_2z] - [x+y|_2z] - [y|_2z] - [x|y|z] + [x|z|y] - [z|x|y] \\
	\partial_4 [x|_2y|z] & = -[x|_2y] + [x|_2y+z] - [x|_2z] + [x|y|z] - [y|x|z] + [y|z|x] \\
	\partial_4 [x|_3y] & = -[x|_2y] - [y|_2x] \\
	\partial_4 [x] & = - [x|_2x]
\end{aligned}
\end{equation*}

If $F$ is an fpqc sheaf, then we obtain a spectral sequence by applying $\operatorname{Hom}(L(G), -)$ to an injective resolution of $F$:
\begin{equation} \label{E:3}
	\operatorname{Ext}^p(L_{q+1}(G), F) \Rightarrow \operatorname{Ext}^{p+q}(G, F) \qquad \text{(for $p+q \leq 3$)}
\end{equation}
When $G$ is a group scheme
, $\operatorname{Ext}^p(\mathbb Z[G^r], F) = H^p(\mathrm{fpqc}(G^r), F)$, and when $F$ is quasicoherent, $H^p(\mathrm{fpqc}(G^r), F) = H^p(\mathrm{zar}(G^r), F)$.

\subsection{Extensions of commutative group schemes by the additive group} \label{S:2}

\begin{pro} \label{P:10}
	Suppose that $k$ is a field and $G$ is an abelian variety over $k$.  Then $\operatorname{Ext}^2_k(G, \mathbb G_a) = 0$.
\end{pro}
\begin{proof}
	We will see that the spectral sequence~\eqref{E:3} degenerates to $0$ in the relevant entries at the $E_2$ page.  The first step is $E_2^{2,0}$, which is the kernel of
	\begin{equation*}
		\xymatrix{\mu^\ast - p_1^\ast - p_2^\ast : H^2(G, \mathcal O_G) \ar[r]& H^2(G^2, \mathcal O_{G^2}) }
	\end{equation*}
	where $\mu : G \times G \to G$ is the multiplication map.  That is, $E_2^{2,0}$ consists of those elements of $H^2(G, \mathcal O_G)$ that behave homogenously linearly.  But $G$ is an abelian variety, so $H^2(G, \mathcal O_G) = \bigwedge^2 H^1(G, \mathcal O_G)$ and $H^1(G, \mathcal O_G) = \operatorname{Ext}^1(G, \mathcal O_G)$ behaves homogeneously linearly, so $H^2(G, \mathcal O_G)$ behaves homogeneously quadratically.  Therefore $E_2^{2,0}$ consists of elements of $H^2(G, \mathcal O_G)$ that behave simultaneously homogeneously linearly and homgeneously quadratically, hence must be~$0$.

	In other words, suppose that $\omega = \sum \alpha_i \wedge \beta_i$ represents an element of $E_2^{2,0}$.  Let $(a,b) : G^2 \to G$ be the map sending $(x,y)$ to $ax + by$.  Then the linearity says
	\begin{equation*}
		(a,b)^\ast \omega = a p_1^\ast(\omega) + b p_2^\ast(\omega)
	\end{equation*}
	whereas the quadraticity says
	\begin{equation*} \begin{aligned}
		(a,b)^\ast \sum \alpha_i \wedge \beta_i & = \sum (a p_1^\ast(\alpha_i) + b p_2^\ast(\alpha_i)) \wedge (a p_1^\ast(\beta_i) + b p_2^\ast(\beta_i)) \\
		& = a^2 p_1^\ast(\omega) + b^2 p_2^\ast(\omega) + \sum ab p_1^\ast(\alpha_i) \wedge p_2^\ast(\beta_i) + ab p_2^\ast(\alpha_i) \wedge p_1^\ast(\beta_i).
	\end{aligned}
	\end{equation*}
	By the K\"unneth formula, we conclude $a^2 \omega = a \omega$ for all integers $a$, so $\omega = 0$.

	The next piece of the filtration comes from $E_2^{1,1}$, which is the homology of the following complex:
	\begin{equation*}
		\xymatrix{H^1(G, \mathcal O_G) \ar[r]^{\partial_2^T} &
                  H^1(G^2, \mathcal O_{G^2}) \ar[r]^-{\partial_3^T} & H^1(G^3, \mathcal O_{G^3}) \times H^1(G^2, \mathcal O_{G^2}) }
	\end{equation*}
	By the linear behavior of $H^1(G, \mathcal O_G)$, the map $\partial_2^T$ vanishes.  We can write $\partial_3^T$ as a matrix:
	\begin{equation*}
		\begin{pmatrix}
			(\mu \times \mathrm{id})^\ast + p_{12}^\ast - (\mathrm{id} \times \mu)^\ast - p_{23}^\ast
			p_{12}^\ast - p_{21}^\ast
		\end{pmatrix}
	\end{equation*}
	Since $H^0(G, \mathcal O_G) = k$, the K\"unneth formula allows us to identify $H^1(G^r, \mathcal O_{G^r}) = H^1(G, \mathcal O_G)^r$.  Suppose that $(\alpha, \beta) \in H^1(G, \mathcal O_G)^2$.  Since $\mu^\ast$ acts a $p_1^\ast + p_2^\ast$ on $H^1(G, \mathcal O_G)$, we have
	\begin{equation*}
		\partial_3^T(\alpha, \beta) = 
		\begin{pmatrix} 
			(\alpha, \alpha, \beta) + (\alpha, \beta, 0) - (\alpha, \beta, \beta) - (0, \beta, \beta) \\ 
			(\alpha - \beta, \beta - \alpha) 
		\end{pmatrix} 
			= 
		\begin{pmatrix} (\alpha, 0, \beta) \\ (\alpha - \beta, \beta - \alpha) 
		\end{pmatrix}  .
	\end{equation*}
	From the first row, we see that if $\partial_3^T(\alpha, \beta) = 0$ then $(\alpha, \beta) = 0$.  Thus $E_2^{1,1} = 0$.

	The last piece of the filtration comes from $E_2^{0,2}$, which is the homology of another sequence:
	\begin{equation} \label{E:4}
		\xymatrix{\Hom(L(G)_2, \mathbb G_a) \ar[r]^{\partial_3^T} &\Hom(L(G)_3, \mathbb G_a) \ar[r]^{\partial_4^T} &\Hom(L(G)_4, \mathbb G_a)}
	\end{equation}
	But homomorphisms $\mathbb Z[G^r] \to \mathbb G_a$ correspond to not-necessarily-homomorphic maps $G^r \to \mathbb G_a$.  Since $G^r$ is proper, reduced, and connected, maps $G^r \to \mathbb G_a$ are constant, and we may identify $\Hom(\mathbb Z[G^r], \mathbb G_a)$ with $k$. Therefore~\eqref{E:4} is the same as the complex associated with the \emph{trivial} group, which computes $\operatorname{Ext}^2(0, \mathbb G_a)$.  This certainly vanishes.
\end{proof}

\begin{pro} \label{P:11}
	Suppose that $k$ is a field and $G$ is an algebraic torus over $k$.  Then $\operatorname{Ext}^2_k(G, \mathbb G_a) = 0$.
\end{pro}
\begin{proof}
	Since $G$ is affine, $H^p(G, \mathbb G_a) = 0$ for all $p > 0$.  Therefore $\operatorname{Ext}^2(G, \mathbb G_a)$ is computed by the complex $\operatorname{Hom}(L(G)_{\bullet - 1}, \mathbb G_a)$.

	We note that $L(G)_{\bullet}$ contains a copy of the Moore complex $M(G)_\bullet$ with $M(G)_r = \mathbb Z[G^r]$.  The complex $\operatorname{Hom}(M(G)_{\bullet}, \mathbb G_a)$ computes the $G$-equivariant cohomology of $\mathbb G_a$, which vanishes in all degrees by \cite[Exp.~I, Thm.~5.3.3, p.~42]{sga3-1}.  Therefore $\operatorname{Ext}^p(G, \mathbb G_a)$ is computed by the complex $\operatorname{Hom}(L(G) / M(G), \mathbb G_a)$:
	\begin{equation*}
		\xymatrix{0 \ar[r]& 0 \ar[r]& k[t,t^{-1}]^{\otimes 2} \ar[r]^-{\partial_2^T} &k[t,t^{-1}]^{\otimes 3} \times k[t,t^{-1}]^{\otimes 3} \times k[t,t^{-1}]^{\otimes 2} \times k[t,t^{-1}]}
	\end{equation*}
	In degree~$2$, the differential sends a Laurent polynomial $p \in k[t,t^{-1}]^{\otimes 2}$ to
	\begin{equation*}
		\begin{pmatrix}
			p(t_1,t_3) - p(t_1 t_2, t_3) + p(t_2, t_3) \\
			p(t_1,t_2) - p(t_1, t_2 t_3) + p(t_1, t_3) \\
			-p(t_1,t_2) - p(t_2,t_1) \\
			-p(t,t)
		\end{pmatrix} .
	\end{equation*}
	In other words, if $\partial_2^T(p) = 0$ then $p$ is an antisymmetric bilinear map $G \times G \to \mathbb G_a$.  But we know that all linear maps $G \to \mathbb G_a$ are constant, so the same must hold for bilinear maps.  Thus $\operatorname{Ext}^2(G, \mathbb G_a)$ vanishes.
\end{proof}

\begin{pro} \label{P:13}
	Let $G$ be a locally constant sheaf of finitely generated abelian groups over a field $k$.  Then $\operatorname{Ext}^p_{k}(G, \mathbb G_a) = \operatorname{Ext}^p_{\mathrm{ab.gp}}(G, k)$, where the second Ext group is in the category of abelian groups and $G$ is taken to be the associated finitely generated abelian group.  In particular, $\operatorname{Ext}^p_{k}(G, \mathbb G_a)$ vanishes for $p \geq 2$.
\end{pro}
\begin{proof}
  Assume first that $G$ is constant. By the elementary divisors theorem, $G$ has a $2$-term resolution by finitely generated free abelian groups, $F_\bullet$.  If $F_i$ is the free abelian group on a set $S$ then $\operatorname{Ext}^p_k(F_i, \mathbb G_a) = H^p(S, \mathbb G_a)$, and this vanishes for $p > 0$.  Therefore $\operatorname{Ext}^p_k(G, \mathbb G_a)$ is computed by $\operatorname{Hom}(F_\bullet, \mathbb G_a)$, which also computes $\operatorname{Ext}^p_{\mathrm{ab.gp}}(G, k)$.
  
	If $G$ is only locally constant, we have a spectral sequence
	\begin{equation*}
		H^p(\mathrm{fpqc}(k), \underline{\operatorname{Ext}}^q(G, \mathbb G_a)) \Rightarrow \operatorname{Ext}^{p+q}_{\mathrm{fpqc}(k)}(G, \mathbb G_a)
	\end{equation*}
	that stabilizes to $0$ at the $E_2$ page by the calculation in the constant case, above.
\end{proof}

\begin{pro}[Main vanishing] \label{P:2}
	Let $G$ be a smooth proper commutative group scheme over a field $k$.  Then $\operatorname{Ext}^2_k(G, \mathbb G_a) = 0$.
	\end{pro}
\begin{proof}
	Let $G^\circ$ be the connected component of $G$; by assumption, this is an abelian variety.  There is an exact sequence
	\begin{equation*}
\xymatrix{		0 \ar[r]& G^\circ \ar[r]& G \ar[r]& \pi_0(G) \ar[r]& 0}
	\end{equation*}
	with $\pi_0(G)$ a finite \'etale group scheme (see \cite[Lem.~2.1]{AHP16}).
	As all finite \'etale group schemes over a field become isomorphic to products of group schemes $\mathbb Z/n\mathbb Z$ after a separable field extension, $\pi_0(G)$ is 
	locally constant as a sheaf on the fppf site. 	
	It follows that $\operatorname{Ext}^2_k(G, \mathbb G_a) = 0$, since $\operatorname{Ext}^2_k(G^\circ, \mathbb G_a) = 0$ (\Cref{P:10}) and $\operatorname{Ext}^2_k(\pi_0(G), \mathbb G_a) = 0$ (\Cref{P:13}). (Another proof can be given using Chevalley's Theorem; see the proof of \Cref{L:Ext2(Gg)'}.)
    \end{proof}

\subsubsection{Comparing \Cref{P:2} to results in the literature}\label{S:Oort}
\Cref{P:2} is related to some vanishing results in the literature.   
 The main point for us is that the vanishing results in the literature are for extension groups in different categories than are needed for Illusie's results on deformation theory.  We found it easier to prove the vanishing results we needed directly, but for clarity, we include here the related, standard result, that exists in the literature: 
  
\begin{lem}[{\cite{oortCGS}}]\label{L:Ext2(Gg)'}
Let $G_0$ be a sub-group scheme of an abelian variety over an algebraically closed field $k$.  Then $\operatorname{Ext}_{\mathrm{ab.gp.sch.}}^2(G_0,\mathbb G_{a,k})=0$; the category in which the extensions are taken is explained in the proof below.
 \end{lem}

\begin{proof}
From Chevalley's Theorem \cite[Thm.~1, p.243]{BLR}, it follows that $G_0$ is an extension of an abelian variety $A_0$ by a finite commutative group scheme $F_0$:
\begin{equation}\label{E:extAF}
\xymatrix{0\ar[r]& F_0\ar[r]& G_0\ar[r]& A_0\ar[r]& 0.}
\end{equation}
The long exact sequence for $\operatorname{Ext}$ gives an exact sequence
$$
\xymatrix{\operatorname{Ext}_{\mathrm{ab.gp.sch.}}^2(A_0,\mathbb G_{a,k})\ar[r]& \operatorname{Ext}_{\mathrm{ab.gp.sch.}}^2(G_0,\mathbb G_{a,k})\ar[r]& \operatorname{Ext}_{\mathrm{ab.gp.sch.}}^2(F_0,\mathbb G_{a,k}).}
$$
On the one hand we have $\operatorname{Ext}_{\mathrm{ab.gp.sch.}}^2(A_0,\mathbb G_{a,k})=0$ (e.g., \cite[Lem.~II.12.8]{oortCGS}).  On the other hand, we have $\operatorname{Ext}_{\mathrm{ab.gp.sch.}}^2(F_0,\mathbb G_{a,k})=0$ \cite[II Cor.~11.10]{oortCGS}.   This reference may take some work to unpack, and so we sketch the necessary steps to do this here.  
First, to be precise,  \cite[II Cor.~11.10]{oortCGS} states that for any $N_0$ in the category $\underline P_0$, the pro-finite category of finite commutative group schemes over $k$  \cite[p.~II.6-4]{oortCGS}, one has $\operatorname{Ext}^2(N_0,\mathbb G_{a,k})=0$
 in the category $\underline P$, 
the pro-finite category of commutative group schemes of finite type over $k$ \cite[p.~II.6-4]{oortCGS}.  On \cite[p.~I.4-1]{oortCGS} it is explained that the category $\underline G$ of commutative group schemes of finite type over $k$ is a sub-category of the category $\underline P$, and the same argument implies that the category $\underline N$ of finite commutative group schemes of over $k$ \cite[p.~II.6-1]{oortCGS} is a sub-category of 
 the category $\underline P_0$. 
Finally, 
 on  \cite[p.~I.4-3]{oortCGS} it is explained that extensions in the category $\underline P$  agree with Yoneda extensions in the category $\underline G$.  
      \end{proof}

\begin{rem}
Comparing the strategy of proof of \Cref{L:Ext2(Gg)'} with the strategy of proof of \Cref{P:2},  we note that in addition to
  the extension \eqref{E:extAF}, we also have an extension
\begin{equation}\label{E:extFA}
\xymatrix{0\ar[r]& A'_0\ar[r]& G_0\ar[r]& F'_0\ar[r]& 0.}
\end{equation}
For this, let $A'_0$ be the maximal abelian subvariety of $G_0\subseteq X_0$, where $X_0$ is the abelian variety containing $G_0$.  Then we have $G_0/A'_0\subseteq X_0/A'_0$, and $G_0/A'_0$ must have dimension $0$.   One can use \eqref{E:extFA} in place of \eqref{E:extAF}  in the proof of \Cref{L:Ext2(Gg)'} above.
\end{rem}

\begin{rem}
For context, we remind the reader that  every finite flat commutative group scheme over $S$ is Zariski locally on $S$ a sub-group scheme of a projective abelian scheme \cite[Thm.~3.1.1, p.110 (Raynaud)]{berthelotbreenmessing}. In particular, for every finite commutative group scheme $G_0$  over an algebraically closed field $k$, we have $\operatorname{Ext}_{\mathrm{ab.gp.sch.}}^2(G_0,\mathbb G_{a,k})=0$. 
\end{rem}

\subsection{Deformations of smooth proper commutative group schemes} \label{S:1}

\begin{cor}  \label{L:Ext2(Gg)}
	Let $G$ be a smooth  proper commutative group scheme over an Artinian local ring $R$.  Let $R'$ be an infinitesimal extension of $R$.  Then $G$ extends to a smooth proper commutative  group scheme over $R'$.
	 \end{cor}
\begin{proof}
	We may assume without loss of generality that the kernel of $R' \to R$ is isomorphic to the residue field $k$.  Then \Cref{T:1} shows that the obstruction to the existence of a flat deformation of $G$ lies in $\operatorname{Ext}^2_{R}(G, \ell_G^\vee \otimes_R k)$.  Since $G$ is smooth, $\ell_G^\vee$ coincides with the Lie algebra $\mathfrak g$ of $G$, and $\ell_G^\vee \otimes_R k = \mathfrak g_0$, the Lie algebra of $G_0 = G \otimes_R k$.  By \Cref{L:R2gg0}, $\operatorname{Ext}^2_{R}(G, \mathfrak g_0) = \operatorname{Ext}^2_{k}(G_0, \mathfrak g_0)$.  This vanishes by \Cref{P:2}.
\end{proof}

\begin{rem}[Deformations of abelian varieties as group schemes are unobstructed] \label{R:DefAbUnObst}
In the case of \Cref{L:Ext2(Gg)} where $G_0=A_0$ is an abelian variety over an algebraically closed field $k$,  this recovers the fact that deformations of abelian varieties, as commutative algebraic groups (i.e., \emph{not} as \emph{polarized} abelian varieties), are unobstructed.  
 The standard argument for $\operatorname{char}(k)=p>0$  is as follows: 

\begin{enumerate}[label=(\alph*)]

\item \label{R:DefAbUnObsta} The Serre--Tate theorem (e.g., \cite[Thm.~1.2.1]{katz-serretate})
  says that deformations of an abelian variety are the same as
  deformations of its $p$-divisible group.

More precisely, let  $R$ be a ring in which $p$ is nilpotent; let $I\subseteq R$ be a
  ideal; and let $R_0 = R/I$.  Then there is an equivalence of
  categories between:
  \begin{itemize}
  \item The categories of abelian varieties over $R$, and
    \item The categories of data $(X_0, X[p^\infty])$ where $X_0$ is
      an abelian scheme over $R_0$, and $X[p^\infty]$ is a deformation
      to $R$ of $X_0[p^\infty] = \varinjlim_n X_0[p^n]$, the
      $p$-divisible group of $X_0$.
    \end{itemize}

\item   If $G_0/k$ is a $p$-divisible group over a perfect field $k$ of codimension $c$ and dimension $d$, then
  the deformation functor of $G_0$ is formally smooth, and
  represented by the formal spectrum  $\operatorname{Spf}W(k)\powser{t_{ij}}_{1 \le i \le
    c, 1 \le j \le d}$ \cite[Cor.~4.8]{illusie-bt}.

\end{enumerate}
In characteristic $0$, one may argue by deforming the abelian variety together with a chosen polarization; however this approach does not work in positive characteristic (see \Cref{R:ObstDefPolAV} below for more discussion of this).  
\end{rem}

\begin{rem}
	Oort gives another argument~\cite[Thm.~(2.2.1), p.~273]{oortFGS} (attributed to Grothendieck) for the unobstructedness of abelian varieties as commutative algebraic groups.  This argument proceeds by filtering the obstruction using the spectral sequence~\eqref{E:3} and using certain symmetries possessed by the obstruction class to establish its vanishing.  The proof of \Cref{L:Ext2(Gg)} was adapted from this argument.
\end{rem}

  \begin{rem}[Deformations of \emph{polarized} abelian varieties may be obstructed] 
    \label{R:ObstDefPolAV}
	  In contrast to the case of deformations of abelian varieties as group schemes, we recall that deformations of \emph{polarized} abelian varieties (deformations of a pair $(A_0,\lambda_0)$ where $A_0$ is an abelian variety over an algebraically closed field $k$ and $\lambda_0$ is a polarization on $A_0$) may be obstructed in positive characteristic.

For any algebraically closed field $k$ of characteristic $p>0$,  there
exist a polarized abelian variety $(A,\lambda)$
 over  $k$,  and a small extension $0\to I\to R'\to R\to 0$ of local
Artinian rings over $k$ with residue field $k$, such that $(A,\lambda)$
lifts to $A_R$, but not to $A_{R'}$.
 This comes down to the fact that if $p^2\mid d$ and $g>1$, then $\mathcal A_{g,d,k}$ has $k$-points with multiple components passing through it, or alternatively,  that $\mathcal A_{g,d,k}$ is non-reduced (see \cite[Rem.~1.13 2), 3)]{djAV}).  

	  We note that deformations of \emph{separably} polarized abelian varieties are unobstructed \cite[Thm.~(2.4.1) (Grothendieck), p.~286 {[22]}]{oortFGS} (so, in particular, deformations of polarized abelian varieties are unobstructed in characteristic~$0$). 
	 Note that the converse may fail: there are polarized abelian varieties $(A_0,\lambda_0)$ where $\lambda_0$ is not separable, but the deformations are unobstructed. 
	  For example, over an algebraically closed field $k$ of characteristic $p > 0$, take $(E_0,p\lambda_0)$ where $\lambda_0$ is a principal polarization on an ordinary elliptic curve $E_0$.  
	   	                   
Note that all of the examples in \cite{norman-algorithm} giving rise to abelian varieties with obstructed deformations come from
    various inseparable polarizations on $E_0 \times_k E_0$, where $E_0$
    is a supersingular elliptic curve over an algebraically closed field $k$.  But if $\lambda_0$ is the
    principal polarization on $E_0$, then the deformation space of
    $(E_0\times E_0, \lambda_0\times \lambda_0)$ is formally
    smooth. So the pathologies of deforming these abelian varieties really come from trying to deform the 
     underlying abelian variety together with its polarization; the deformations of the abelian variety as an abstract commutative group scheme are unobstructed.
\end{rem}

  \begin{exa}[Deformations of finite commutative group schemes may be obstructed]
    \label{R:exobstructed} 
	  \Cref{L:Ext2(Gg)} shows that deformations of subgroup schemes of abelian varieties are unobstructed in characteristic~$0$.  However, subgroups of abelian varieties can have obstructed deformations in mixed characteristic.

  For example, let $p$ be a prime and let $R$ be a local ring with
  residue field $k := \bar{\mathbb F}_p$ such that the prime $p$ does not lie in
  the maximal ideal of $R$.
    Then the group scheme $\aalpha_p/k$ (see \S\ref{SS:fingroupschemes} for a reminder)
    does not lift to $R$ \cite[Exa.~(--A), p.317]{OM68};  note however, that is is also shown there that $R'=R[t]/(t^2-p)$, a degree $2$ cover of $R$ ramified at $p$,  gives a faithfully flat morphism $R\to R'$ such that $G$ lifts to $R'$.

	  This example shows in particular that
          $\operatorname{Ext}^2_k(\aalpha_p, \ell_{\aalpha_p}^\vee)
          \neq 0$.  We may identify $\ell_{\aalpha_p} = [ k
            \xrightarrow0 k ]$
             in degrees $[-1,0]$, so this shows that
          either $\operatorname{Ext}^2_k(\aalpha_p, \mathbb G_a) \neq
          0$ or $\operatorname{Ext}^1_k(\aalpha_p, \mathbb G_a) \neq
          0$.  In fact $\operatorname{Ext}^1_k(\aalpha_p, \mathbb G_a)
          \iso k$ (see the proof of \Cref{P:5}, or \cite[Cor.~11.12]{oortCGS})
          and $\operatorname{Ext}^2_k(\aalpha_2, \mathbb G_a)$ is known to be nonzero if $k$ has characteristic~$2$ \cite{Breen1969a}.
   
\end{exa}

\section{Strengthening \Cref{T:KerMorph}}\label{A:Strong}

We now prove a stronger version of \Cref{T:KerMorph}:

\begin{teo}\label{T:KerMorph-Strong}
	Let $S$ be a scheme and let 
$
f:X\to Y
$
be a morphism of abelian schemes over $S$ with kernel $\ker(f)$. 
         If, for each geometric point $s$ of $S$, we have that  $\Hom(\ker(f)_s, \mathbb G_a) = 0$, where $\ker(f)_s$ is the fiber over $s$, then $\ker (f)$ is flat over $S$.
\end{teo}

Our strategy for proving \Cref{T:KerMorph-Strong} is to  prove a statement about deformations of kernels of homomorphisms smooth commutative group schemes: 

\begin{pro}[Obstruction to Flatness] \label{P:3}
Let $R$ be an Artinian local ring with residue field  $k$.  
	  Let $R' \to R$ be a small extension of $R$  by a $k$-vector space $I$.  
	Suppose that $f' : X' \to Y'$ is a homomorphism of  smooth commutative algebraic group schemes over $R'$, restricting to $f : X \to Y$ over $R$ and to $f_0 : X_0 \to Y_0$ over the residue field $k$. 
	Let $\mathfrak h$ be the co-kernel of the morphism $\mathfrak x\to \mathfrak y$ of Lie algebras associated to $f_0:X_0\to Y_0$.  
 Assume that $\ker(f)$ is flat over $R$. 
 Then
 there is a natural obstruction in 
		 $\operatorname{Ext}^0_k(\ker(f_0),\mathfrak h)$,  
		 i.e., a homomorphism of fpqc sheaves of abelian groups, $\ker(f_0) \to \mathfrak h$,		 whose vanishing is equivalent to 
		 the flatness of $\ker(f')$ over $R'$.
 \end{pro}

In fact, our proof comes down to two statements, parallel to the  obstruction and tangent statements Illusie proves: 
  	\begin{enumerate}[label=(\alph*)]
		\item   There is a natural obstruction in 
		 $$\operatorname{Ext}^0_k(\ker(f_0),\mathfrak h)$$
				 whose vanishing is equivalent to 
		  		 the existence of a flat extension of $\ker (f)$ over $R'$ contained in $\ker (f')$.

 		\item   If this obstruction vanishes, there is exactly one such flat extension, and it is equal to $\ker (f')$.
  	\end{enumerate}
 
\begin{rem} 
	With more care, the hypotheses on $X$ and $Y$ can probably be weakened.  For example, it may be sufficient to assume that $X \to Y$ is a local complete intersection morphism, with no smoothness assumptions.
\end{rem}

\begin{rem}\label{R:15}
As motivation for \Cref{P:3}, by taking a careful look at the proof of \Cref{T:KerMorph} (\S \ref{S:Pf1}),
 in the situation of the small extension of Artinian local rings, then under the assumption 
 $\operatorname{Ext}^2_k(G_0, \mathbb G_a)=0$, e.g., if  $G_0$ is a smooth proper group scheme over a field (\Cref{P:13}), 
  one can see that even without the assumption that $\operatorname{Hom}(G_0,\mathbb G_a)=0$, that the class $\alpha \in \operatorname{Ext}^0_k(G_0,\mathfrak y)$ given by the difference of the lifts $G''\to X'\to Y'$ and $0:G''\to Y'$ is an obstruction to the former composition being equal to the zero morphism. Allowing ourselves to modify the morphism $G''\to X'$ by elements of $\operatorname{Ext}^0_k(G_0,\mathfrak x)$, we see that it is the image of $\alpha$ in $$\operatorname{Ext}^0_k(G_0,\mathfrak h)$$ under the natural map $\operatorname{Ext}^0_k(G_0,\mathfrak y)\to \operatorname{Ext}^0_k(G_0,\mathfrak h)$ whose vanishing is \emph{equivalent} to the flatness of the kernel $G'$.  In fact, one can see that this observation, together with the proof of \Cref{T:KerMorph}, gives a proof of \Cref{P:3} so long as one assumes that $\operatorname{Ext}^2_k(G_0,\mathbb G_a)=0$.  
   Note that to conclude in the proof that $G'' \to X'$ is a closed immersion, we use that $G_0$ is a local complete intersection over $k$ \cite[III, \S3, no.~6, Thm.~6.1, p.~346]{DemazureGabriel} (see also [SGA 3 III (new version), Prop.~4.15]: every flat locally finite presentation sub-group scheme of a smooth group scheme is a local complete intersection in the ambient group scheme).
   In summary, this gives a proof of \Cref{P:3} under added assumption that $X'$ is proper and $\ker (f)$ is smooth. 
\end{rem}

Before  proving \Cref{P:3}, we explain how  \Cref{P:3} implies \Cref{T:KerMorph-Strong}:

\begin{proof}[Proof of \Cref{T:KerMorph-Strong}] 
	By \Cref{L:R2Artin}, we may assume that $S$ is Artinian.  By induction on the length of $S$, it will be sufficient to assume the result is already known for $S$ and deduce it for a small extension $S'$ (\ref{SSS:small}).  Let $f' : X' \to Y'$ be a morphism of abelian schemes over $S'$ and let $f : X \to Y$ be its restriction to $S$.  Let $\ker(f_0)$ be the restriction of $\ker(f)$ to the residue field of $S$.  We are in the situation of \Cref{P:1}, so there is a homomorphism $\ker(f_0) \to \mathfrak h$ obstructing the flatness of $f'$ (using the notation of \Cref{P:1}).  But, by assumption, we have $\operatorname{Hom}(\ker(f_0),\mathfrak h)=0$,  so the obstruction $\ker(f_0) \to \mathfrak h$ vanishes,  and $\ker(f')$ is flat over $S'$.
\end{proof}

\subsection{Proof of \Cref{P:3}} \label{S:6}

\begin{lem} \label{L:4}
	Let $f : X \to Y$ be a morphism of schemes and let $Z$ be a locally closed subscheme of $Y$.  Then the morphism on cotangent complexes
	\begin{equation} \label{E:1}
		\xymatrix{L_{X/Y} \big|_{f^{-1} Z} \ar[r]& L_{f^{-1} Z/Z}}
	\end{equation}
	is an isomorphism on homology in degree~$0$ and surjective on homology in degree~$-1$.
\end{lem}
\begin{proof}
	Let $K$ be the cone of~\eqref{E:1}.  The assertion is equivalent to the vanishing of $H^0(K)$ and $H^{-1}(K)$, which is equivalent to the vanishing of $\ext^0(K \big|_U,J)$ and $\ext^1(K \big|_U,J)$ for all affine schemes $U$, morphisms $U \to f^{-1} Z$, and injective quasicoherent $\mathcal O_U$-modules $J$.  This is equivalent to saying that the map
	\begin{equation*}
\xymatrix{		\ext^i(L_{f^{-1} Z/Z} \big|_U, J) \ar[r]& \ext^i(L_{X/Y} \big|_U, J) }
	\end{equation*}
	is an isomorphism for $i = 0$ and injective for $i = 1$.  If we write $\mathbf{Ext}(-,-)$ for the \emph{groupoid} of extensions, these two conditions are equivalen to the full faithfulness of
	\begin{equation*}
		\xymatrix{\mathbf{Ext}(L_{f^{-1} Z/Z} \big|_U, J) \ar[r]& \mathbf{Ext}(L_{X/Y} \big|_U, J).}
	\end{equation*}
	We may identify $\mathbf{Ext}(L_{f^{-1} Z/Z} \big|_U, J)$ with the groupoid of $\mathcal O_Z$-algebra extensions of $\mathcal O_{f^{-1} Z} \big|_U$ by $J$.  Likewise, $\mathbf{Ext}(L_{X/Y} \big|_U, J)$ is the groupoid of $\mathcal O_Y$-algebra extensions of $\mathcal O_X \big|_U$ by $J$.  We must therefore show that if $A$ is an $\mathcal O_Z$-algebra extension of $\mathcal O_{f^{-1} Z} \big|_U$ by $J$ and $B$ is the induced $\mathcal O_Y$-algebra extension of $\mathcal O_X \big|_U$ by $J$, then every splitting of $B$ is induced from a unique splitting of $A$.  Visually, we have a commutative diagram of solid arrows with exact rows:
	\begin{equation} \label{E:2} \vcenter{ \xymatrix{
			0 \ar[r] & J \ar[r] \ar@{=}[d] & B \ar[r] \ar[d] & \mathcal O_X \big|_U \ar[r] \ar[d] \ar@{-->}[dl]_{\varphi} & 0 \\
			0 \ar[r] & J \ar[r] & A \ar[r] & \mathcal O_{f^{-1} Z} \big|_U \ar[r] & 0
		}}
	\end{equation}
	We must show that every splitting of the top row is induced from a splitting of the bottom row.  A splitting of the top row is a $\mathcal O_Y$-algebra homomorphism $\mathcal O_X \big|_U \to B$ that splits the projection.  This induces a map $\varphi : \mathcal O_X \big|_U \to A$ as illustrated with the dashed arrow in \Cref{E:2}.  But $A$ is an $\mathcal O_Z$-algebra extension of $\mathcal O_{f^{-1} Z}$, so, if $I$ is the ideal of $\mathcal O_Z$ in $\mathcal O_Y$, then $\varphi(I \mathcal O_X \big|_U) = 0$.  That is, $\varphi$ factors uniquely through $\mathcal O_X \big|_U / I \mathcal O_X \big|_U = \mathcal O_{f^{-1} Z} \big|_U$, as required.
\end{proof}

\begin{cor} \label{L:1}
	Let $X$ and $Y$ be flat, commutative group schemes over a scheme $S$.  Let $f : X \to Y$ be a homomorphism over $S$.  Let $G$ be the kernel of $f$.  Then the cone of
\begin{equation} \label{eqn:1}
	\ell_{X/Y} \longrightarrow \ell_G
\end{equation}
	is concentrated in degrees $[-3,-2]$.  If $X \to Y$ is a local complete intersection morphism (for example if $Y$ is smooth) then the cone is concentrated in degree~$-2$.
\end{cor}
\begin{proof}

	Let $K$ be the cone of $L_{X/Y} \big|_G \to L_{G/S}$.  By \Cref{L:4}, $K$ is concentrated in degrees $\leq -2$.  But both $X$ and $Y$ are local complete intersections over $S$ (by \cite[Exp.~VIIB, Cor.~5.5.1]{sga3-1} or \cite[III, \S3, no.~6, Th\'eor\`eme~6.1, p.~346]{DemazureGabriel}), so $L_{X/S}$ and $L_{Y/S}$ are both concentrated in $[-1,0]$.  The exactness of the triangle
	\begin{equation*}
		\xymatrix{L_{X/S} \ar[r]& L_{Y/S} \ar[r]& L_{X/Y}}
	\end{equation*}
	implies that $L_{X/Y}$ is concentrated in $[-2,0]$, so $K$ is concentrated in $[-3,0]$.  Combining these observations, we conclude that $K$ is concentrated in $[-3,0] \cap (-\infty,-2] = [-3,-2]$.  Restricting to the origin of $G$, we deduce that the cone of~\eqref{eqn:1} is also concentrated in degree $[-3,-2]$.

	If $X \to Y$ is a local complete intersection morphism then $L_{X/Y}$ is concentrated in degrees $[-1,0]$, and therefore $K$ is concentrated in degree~$-2$, as is its restriction to the origin of $G$.
\end{proof}

\Cref{L:1}  shows that, if $Y$ is smooth over $k$, then the cone of $\ell_G^\vee \to \ell_{X/Y}^\vee$ is quasi-isomorphic to a vector space in degree~$1$.

\begin{pro} \label{P:1}
	Let $S_0 \subseteq S \subseteq S'$ be infinitesimal extensions of Artinian local rings, with $S_0 = \operatorname{Spec} k$ and with the ideal $J$ of $S$ in $S'$ isomorphic to $k$.  Suppose that $f' : X' \to Y'$ is a homomorphism of commutative group schemes over $S'$, restricting to $f : X \to Y$ over $S$ and to $f_0 : X_0 \to Y_0$ over $S_0$.  Assume that $X'$ is smooth over $S'$.  Let $G'$ be the kernel of $f'$ and let $G$ and $G_0$ be its restrictions to $S$ and to $S_0$.  Assume that $G$ is flat over $S$.  Let $\mathfrak h$ be a $k$-vector space such that $\mathfrak h[-1]$ is quasi-isomorphic to the cone of $\ell_{G_0/S_0}^\vee \to \ell_{X_0/Y_0}^\vee$.  Then there is a homomorphism of fpqc sheaves of abelian groups $G_0 \to \mathfrak h$, depending naturally on the above data, whose vanishing is equivalent to the flatness of $G'$ over $S'$.
\end{pro}
\begin{proof}
	First, we construct the map $G \to \mathfrak h$.  Up to gluing, this is a local problem in the Zariski topology of $G$.  Suppose that $p : U \to G$ is a morphism of schemes.\footnote{For now, we only need to consider the case where $U$ is an open subset of $G$, in which case $p^{-1} \mathcal O_G = \mathcal O_U$.  However, the construction works for an arbitrary morphism $U \to G$.}
	Since $G$ is a local complete intersection over $S$ (by \cite[Exp.~VIIB, Cor.~5.5.1]{sga3-1} or \cite[III, \S3, no.~6, Th\'eor\`eme~6.1, p.~346]{DemazureGabriel}), we can, at least after replacing $U$ by an open cover, find a sheaf $\mathcal O_{U'}$ of $\mathcal O_{S'}$-algebras on $U$ fitting into a commutative diagram:\footnote{We have suppressed notation for the sheaf pullbacks of $J$, $\mathcal O_{S'}$, and $\mathcal O_S$ to $U$.}
	\begin{equation*} \xymatrix{
			0 \ar[r] & J \ar[d] \ar[r] & \mathcal O_{S'} \ar[r] \ar[d] & \mathcal O_S \ar[r] \ar[d] & 0 \\
			0 \ar[r] & J \otimes \mathcal O_U \ar[r] & \mathcal O_{U'} \ar[r] & p^{-1} \mathcal O_G \ar[r] & 0
		}
	\end{equation*}
	Let $U'$ be the locally\footnote{It will not matter for us that $U'$ is locally ringed as opposed to merely ringed.  To see it is locally ringed, note that since the kernel of $\mathcal O_{U'} \to p^{-1} \mathcal O_G$ is nilpotent, the local rings of $\mathcal O_{U'}$ and of $p^{-1} \mathcal O_G$ are the same, and $p^{-1} \mathcal O_G$ is a sheaf of local rings because it is the pullback of the structure sheaf of a scheme.}
	ringed space $(U, \mathcal O_{U'})$.  The isomorphism classes of choices of $\mathcal O_{U'}$ form a torsor on $U$ under the sheaf $\ext^1( L_{G/S}, J \otimes \mathcal O_U)$.

	Since $X'$ is smooth over $S'$, we can, again after replacing $U$ by an open cover if necessary, find an $S'$-morphism $g' : U' \to X'$ extending $U \to X$.  The composition $U' \to X' \to Y'$ corresponds dually to an algebra homomorphism:
	\begin{equation} \label{eqn:4}
\xymatrix{		{g'}^\ast {f'}^\ast : e^{-1} \mathcal O_{Y'} \ar[r]& \mathcal O_{U'}}
	\end{equation}
	The symbol $e$ denotes the composition $U' \to S'$ with the inclusion of the origin $S' \to Y'$, which coincides with $f' g'$, \emph{topologically}.  Since ${g'}^\ast {f'}^\ast$ coincides with $e^\ast$ modulo $J$, the difference ${g'}^\ast {f'}^\ast - e^\ast$ factors through a derivation:
	\begin{equation*}
		\xymatrix{\delta : e^{-1} \mathcal O_{Y_0} \ar[r]& J}
	\end{equation*}
	We can view $\delta$ as a homomorphism $\ell_{Y_0} \to J$, or, equivalently, as an element of $H^0(U_0, \ell_{Y_0}^\vee \otimes J)$.  It induces an element $\gamma$ of $\mathfrak h \otimes J$ by composition with the projection $\ell_{Y_0}^\vee \to \mathfrak h$.

	To complete the definition of the map $G \to \mathfrak h \otimes J$, we observe that the definition of $\gamma \in H^0(U, \mathfrak h \otimes J)$ does not depend on the choice of $U'$, nor on the choice of map $g' : U' \to X'$.  More precisely, the image of $\delta$ in $H^1(U_0, \ell_{X_0/Y_0}^\vee \otimes J)$ is independent of the choice of map $g' : U' \to X'$ because the choices of $g'$ form a torsor under $H^0(U_0, \ell_{X_0}^\vee \otimes J)$ and the sequence
	\begin{equation} \label{eqn:8}
		H^0(U, \ell_{X_0}^\vee \otimes J) \to H^0(U, \ell_{Y_0}^\vee \otimes J) \to H^1(U, \ell_{X_0/Y_0}^\vee \otimes J)
	\end{equation}
	is exact.  Similarly, choices of $U'$ correspond to sections of $H^1(U, \ell_{G_0}^\vee \otimes J)$ and the influence of this choice on $H^0(U, \mathfrak h \otimes J)$ is dismissed by another exact sequence:
	\begin{equation} \label{eqn:7}
		\xymatrix{H^0(U, \ell_{G_0}^\vee \otimes J) \ar[r]& H^1(U, \ell_{X_0/Y_0}^\vee \otimes J) \ar[r]& H^0(U, \mathfrak h \otimes J)}
	\end{equation}
	Since $\gamma$ is independent of the choices of $U'$ and of $U' \to X'$, it is compatible with localization in $U$, and therefore glues to a section of $\mathfrak h \otimes J$ over $G$.  Since $\mathfrak h \otimes J$ is supported over $S_0$, this section is induced from a unique section of $\mathfrak h \otimes J$ over $G_0$.  We abuse $\gamma$ slightly and use it to denote this section as well.

	\vskip\baselineskip

	We verify next that $\gamma$ is a homomorphism.  That is, we must show that $\mu^\ast \gamma = p_1^\ast \gamma + p_2^\ast \gamma$ as a section of $\mathfrak h \otimes J$ over $G_0 \mathop\times_{S_0} G_0$, with $\mu$ denoting the group operation of $G_0$.  This is a local question in the Zariski topology of $G$, so we consider a morphism of schemes $p : U \to G \mathop\times_{S} G$ and an $\mathcal O_{S'}$-algebra extension $\mathcal O_{U'}$ of $p^{-1} \mathcal O_{G \mathop\times_S G}$ by $J \otimes \mathcal O_U$, which exists Zariski-locally since $G \mathop\times_{S} G \to S$ is a local complete intersection morphism.  As before, we write $U'$ for the locally ringed space $(U, \mathcal O_{U'})$.  Replacing $U$ by an open cover if necessary, we choose an extension $U' \to X' \mathop\times_{S'} X'$ of $U \to X \mathop\times_S X$, which exists by the smoothness of $X \mathop\times_S X$.  Then $\mu^\ast \gamma$ is induced by composition:
	\begin{equation} \label{eqn:5}
		U' \to X' \mathop\times_{S'} X' \xrightarrow\mu X' \to Y'
	\end{equation}
	The construction of $p_1^\ast \gamma + p_2^\ast \gamma$ uses the group structure of the sheaf $\mathfrak h \otimes J$, which is induced by the group structure of $\mathfrak y \otimes J$, where $\mathfrak y$ is the Lie algebra of $Y$.  We observe that this group structure is the same as the one induced from the group operation of $Y'$.\footnote{Indeed, $\mathfrak y \otimes J$ can be obtained by evaluating the functor of points of $Y'$ on certain schemes supported at the origin of $Y'$.  This induces a group structure on $\mathfrak y$ in two ways: from the group structure of $Y'$ and from pushout of square-zero extensions of the origin.  Since each of these group operations is a homomorphism with respect to the other, they coincide by the Eckmann--Hilton argument.}
	Thus, $p_1^\ast \gamma + p_2^\ast \gamma$ is induced from the composition
	\begin{equation} \label{eqn:6}
		U' \to X' \mathop\times_{S'} X' \to Y' \mathop\times_{S'} Y' \xrightarrow\mu Y' .
	\end{equation}
	Since the compositions~\eqref{eqn:5} and~\eqref{eqn:6} coincide, $\gamma$ is a homomorphism.

	\vskip\baselineskip

	To conclude the proof, we show that $\gamma$ obstructs the flatness of $G'$.  Suppose that $\gamma$ is $0$.  We argue that $G'$ is flat over $S'$.  This is a local question in the Zariski topology of $G$, so we make our usual choices of a morphism of schemes $p : U \to G$, an $\mathcal O_{S'}$-algebra extension $\mathcal O_{U'}$ of $p^{-1} \mathcal O_G$ by $J \otimes \mathcal O_U$, and a morphism $U' \to X'$ extending $U \to X$, with $U'$ being the locally ringed space $(U, \mathcal O_{U'})$.  Then the section $\delta \in H^0(U, \ell_{Y_0}^\vee \otimes J)$ constructed above induces an $\alpha \in H^1(U, \ell_{X_0/Y_0}^\vee \otimes J)$ that, by the exactness of~\eqref{eqn:7}, lies in the subgroup $H^1(U, \ell_{G_0}^\vee \otimes J) \subseteq H^1(U, \ell_{X_0/Y_0}^\vee \otimes J)$ (that this is a subgroup was shown in \Cref{L:1}).  But $H^1(U, \ell_{X_0/Y_0}^\vee \otimes J)$ acts simply transitively on the choices of deformation $U'$ of $(U, p^{-1} \mathcal O_G)$ above.  Replacing $U'$ by $U' - \alpha$, we find that $\delta \in H^0(U, \ell_{Y_0}^\vee \otimes J)$ induces $0$ in $H^1(U, \ell_{X_0/Y_0}^\vee \otimes J)$.\footnote{By \Cref{L:1}, the choice of $U'$ is unique up to unique isomorphism.}

	By the exactness of~\eqref{eqn:8}, $\delta$ lifts to $\beta \in H^0(U, \ell_{X_0}^\vee \otimes J)$.  But $H^0(U, \ell_{X_0}^\vee \otimes J)$ acts simply transitively on the choices of morphism $g' : U' \to X'$ used in the construction of $\delta$.  Replacing $g'$ by $g' - \beta$, we then find that $\delta = 0$.  But $\delta$ measured the deviation between the map $f'g' : U' \to Y'$ and the trivial map $e' : U' \to Y'$, so this means $g' : U' \to X'$ factors through $G'$.

	The reasoning of the previous paragraph applies in particular to all sufficiently small open subsets $U$ of $G$:  each of these admits a flat deformation over $S'$ \emph{inside} of $G'$.\footnote{At this point, we may also conclude by the same argument used to conclude the proof of \Cref{T:KerMorph}.} If $U$ is such an open subset then we have a commutative diagram:
	\begin{equation*} \xymatrix{
			J \otimes \mathcal O_{G'} \big|_{U'} \ar[r] \ar[d] & \mathcal O_{G'} \big|_{U'} \ar[d] \\
			J \otimes \mathcal O_{U'} \ar[r] & \mathcal O_{U'}
		}
	\end{equation*}
	Since the lower horizontal arrow is injective and the left vertical arrow is an isomorphism, this implies that $J \otimes \mathcal O_{G'} \to \mathcal O_{G'}$ is injective, and in particular, that $\operatorname{Tor}_p^{\mathcal O_{S'}}(\mathcal O_S, \mathcal O_{G'}) = 0$ for all $p > 0$.  But we have assumed that $\mathcal O_G$ is flat over $S$, so $\operatorname{Tor}_q^{\mathcal O_S}(\mathcal O_{S_0}, \mathcal O_G) = 0$ for all $q > 0$.  By the spectral sequence
	\begin{equation*}
		\operatorname{Tor}_q^{\mathcal O_S}( \mathcal O_{S_0}, \operatorname{Tor}_p^{\mathcal O_{S'}}(\mathcal O_S, \mathcal O_{G'})) \Rightarrow \operatorname{Tor}_{p+q}^{\mathcal O_{S'}}(\mathcal O_{S_0}, \mathcal O_{G'})
	\end{equation*}
	it now follows that $\operatorname{Tor}_n^{\mathcal O_{S'}}(\mathcal O_{S_0}, \mathcal O_{G'}) = 0$ for all $n > 0$, and therefore that $\mathcal O_{G'}$ is flat over $\mathcal O_{S'}$.
\end{proof}

\appendix

\section{A small observation on images of abelian schemes over Artinian rings}\label{S:ArtinRing}

The primary purpose of this section is to prove \Cref{T:ImIsAbSchArt}, which establishes that over Artinian local rings,  the image of a homomorphism of abelian schemes is an abelian scheme if and only if the image is flat. Recall that this need not hold over other bases (e.g.,  \Cref{E:Serre} and \Cref{E:charp}). 

 The secondary purpose of this section is to explain precisely where an error arose in a previous version of this paper (this error led us to the \emph{erroneous} conclusion that the image of a morphism of abelian schemes was always an abelian scheme); the error arose from assuming that the inclusion \eqref{E:PassTheCapInf} was an equality.  
 We discuss this in more detail in \S \ref{S:error-expl}.
   \Cref{T:ImIsAbSchArt} is the outcome of correcting this error.  

\begin{pro}\label{T:ImIsAbSchArt}
Let $f:X\to Y$ be a homomorphism of abelian schemes over $S=\operatorname{Spec}R$ where $(R,\mathfrak m)$ is an Artinian local ring with algebraically closed residue field.  Then $f(X)\subseteq Y$ is a sub-abelian scheme over $S$ if and only if $f(X)$ is flat over $S$. 
\end{pro}

\begin{rem}[Images of abelian schemes over Artinian rings]\label{R:ImNotStable}
Here we discuss the restriction to Artinian local rings of the examples \Cref{E:Serre} and \Cref{E:charp} of  morphisms of abelian
schemes $f:X\to Y$ over  DVRs $R$ where the image $f(X)$ is  flat over $R$, but is  not an abelian scheme. 
 There exists an Artinian local ring $A$ supported  at the special point of $R$ such that the restriction $f_A:X_A\to Y_A$ gives a morphism of abelian schemes over an Artinian local ring $A$ where the image is not an abelian scheme.  Indeed, if the image were an abelian scheme for every such Artinian local ring, then \Cref{L:R2Artin} would imply that $f(X)$ was an abelian scheme.  Let us consider this example further.  Picking such an Artinian ring, and having established that  $f_A(X_A)$ is not an abelian scheme, then by \Cref{T:ImIsAbSchArt}, we have that $f_A(X_A)$ is not a flat group scheme.  Consequently, since $f(X)_A$ is flat by base change, we have that the natural inclusion $f_A(X_A)\subseteq f(X)_A$ (e.g., \cite[p.216]{EH00})  is not an equality.  
\end{rem}

The key point we will use is the following technical proposition, whose proof we postpone until \S\ref{S:ProofFlatRedFib}. 

 \begin{pro}
    \label{P:FlatRedFib}
    Let $Z$ be a scheme over $S=\operatorname{Spec}R$ where $(R,\mathfrak m)$ is a local Artinian ring, and  suppose there is a collection of closed subschemes $W_n \subseteq Z$ such that
  \begin{enumerate}[label=(\alph*)]
   \item   $W_n$ is finite and flat over $S$ with reduced  special fiber, 
   \item  $W_m \cap W_n = \emptyset$ for $m\ne n$, 
   \item and the collection $\st{W_n}$ is schematically dense in $Z$; i.e., $\overline {\bigcup W_n} = Z$.
   \end{enumerate}
    Then if $Z$ is flat over $S$, then it has reduced special fiber.
\end{pro}  

Before proving  \Cref{P:FlatRedFib}, we use it to prove \Cref{T:ImIsAbSchArt}.

\begin{proof}[Proof of \Cref{T:ImIsAbSchArt}]
  Let $Z=f(X)$ be the schematic image of $f$.  Using that $f$ is a   group homomorphism, together with the universal property of the schematic image, one can show that $Z$ is an $S$-group scheme.  Moreover, $Z$
  is proper over $S$, being a closed subscheme of $Y$.   We also claim that $Z$ has connected central fiber.  Indeed, let $s$ be the closed point of $S$; there is always an inclusion $f_s(X_s) \subseteq (f(X))_s$ (e.g., \cite[p.~216]{EH00}), however, since $f$ is proper, one has that the support of the two schemes is the same (e.g., \cite[p.~218]{EH00}).  Therefore, since $X_s$ is connected, so is $f(X_s)$, and therefore, since connectedness is a statement about the support, we have that $(f(X))_s$ is also connected.
  Since a
  geometrically reduced group scheme is smooth, to show
  that $Z$ is an abelian scheme, it now
  suffices to show that $Z$ is flat with reduced special fiber (the residue field is assumed to be algebraically closed).

To this end, choose  a prime $\ell$ that is invertible in $R$, and consider the closed subschemes $X[\ell^n]\hookrightarrow X$ and $Y[\ell^n]\hookrightarrow Y$ for all $n$ .  
It is a basic fact  (e.g.,  \cite[Proof of Thm.~3.19, p.54]{conradtrace}) that 
$$
X = \overline {\bigcup X [\ell^n]}.
$$ 
  Moreover, from our choice of $\ell$, the $X [\ell^n]$   are proper \'etale group schemes over $S$, and since $R$ is an Artinian local ring with algebraically closed residue field, each of 
$X [\ell^n]$ and $Y[\ell^n]$ consist of irreducible components canonically isomorphic to $S$.  The restricted morphism $f[\ell^n]\colon X [\ell^n]\to Y[\ell^n]$ is a morphism over $S$, and therefore, on each irreducible component $f[\ell^n]$ is the identity (although some components of $X [\ell^n]$ may map to the same component of $Y[\ell^n]$).  Consequently, the kernel $\ker f[\ell^n]$  of $f[\ell^n]\colon X [\ell^n]\to Y[\ell^n]$  is a proper \'etale group scheme over $S$, and the quotient  $X [\ell^n]/\ker f[\ell^n]$ (see, e.g., \Cref{T:quot}) is a proper \'etale group scheme over $S$.  Clearly we have the agreement $f(X  [\ell^n])= f[\ell^n](X [\ell^n])= X [\ell^n]/\ker f[\ell^n]$.   
Since scheme theoretic images and closures commute (the scheme theoretic closure is the scheme theoretic image of the morphism from the disjoint union of the closed subschemes),  we have 
\begin{equation}\label{E:pr:f(X)11}
Z:= f(X )= f(\overline {\bigcup X [\ell^n]})= \overline {\bigcup f(X [\ell^n])}
\end{equation}
so that $Z$ is the scheme theoretic closure of the proper \'etale  group schemes $f(X [\ell^n])$.

We now want to invoke \Cref{P:FlatRedFib}.  For this, let $W_n:= f(X[\ell^n])\smallsetminus f(X[\ell^{n-1}])$.  These clearly satisfy the conditions in the proposition, and therefore, assuming $Z$ is flat over $S$, we see it has reduced special fiber, completing the proof.
\end{proof}

\subsection{Proof of \Cref{P:FlatRedFib}}\label{S:ProofFlatRedFib}

Because scheme theoretic closures are Zariski local (scheme theoretic image is stable under flat base change \cite[Prop.~V.8]{EH00}), it suffices to prove the following affine version of \Cref{P:FlatRedFib}: 

\begin{pro}
    \label{P:FlatRedFibAff}
    Let $A$ be an $R$-algebra, where $(R,\mathfrak m_R)$ is a local Artinian ring, and  suppose there is a collection of  ideals $I_n\subseteq A$  such that
  \begin{enumerate}[label=(\alph*)]
   \item  \label{P:FlatRedFibAff-red} $A/I_n$ is a finite flat $R$-module with $(A/I_n)\otimes_R (R/\mathfrak m_R)$ reduced, 
   \item  \label{P:FlatRedFibAff-cop}$I_m + I_n = (1)$ for $m\ne n$, and, 
   \item \label{P:FlatRedFibAff-dense}  $\bigcap I_n = (0)$.
   \end{enumerate}
    If $A$ is flat over $R$, then $A\otimes_R (R/\mathfrak m_R)$ is reduced.
\end{pro}

\subsubsection{The local flatness criterion}
We start by recalling the local flatness criterion, in the form we will use it.
Recall (see also \ref{SSS:small}) that a local ring $(R,\mathfrak m_R)$ is a small extension of a ring $R_0$ if there is an exact sequence
\begin{equation}
  \label{E:small}
  \xymatrix{
    0 \ar[r] & (\epsilon) \ar[r] & R \ar[r] & R_0 \ar[r] & 0
  }
\end{equation}
with $\mathfrak m_R \epsilon = 0$.  Note that by assumption $\epsilon$ is not a unit, and so is contained in $\mathfrak m_R$, and the condition $\mathfrak m_R\epsilon = 0$ implies in particular that $\epsilon^2=0$, so that $(\epsilon)$ is nilpotent.  We will use the following well-known flatness criterion:

\begin{lem}
\label{L:flat criterion}
  Let $M$ be an $R$-module with $(R,\mathfrak m)$ a local ring.
  \begin{enumerate}[label=(\alph*)]
\item    Then $M$ is flat if and only if the natural map
  \[
    \xymatrix{\mathfrak m_R \otimes_R M \ar[r]& M}
  \]
  is an injection.

\item Suppose $R$ is a small extension of $R_0$ as in \eqref{E:small}.  Then $M$ is flat if and only if
  \begin{enumerate}[label=(\roman*)]
  \item $M\otimes_R R_0$ is flat over $R_0$, and
  \item $(\epsilon) \otimes_R M \to M$ is injective.
  \end{enumerate}
\end{enumerate}
\end{lem}

\begin{proof}
 In an Artinian local ring, the maximal ideal $\mathfrak m_R$ is nilpotent, and so  \cite[Thm.~22.3(3)]{Mat89} gives (a); similarly, applying that theorem to the nilpotent ideal $(\epsilon)$ gives (b).
   \end{proof}

\begin{lem}
\label{L:trivial tensor}
Suppose $R$ is a small extension of $R_0$ as in \eqref{E:small} and
that $B$ is an $R$-algebra.  Then
any element of $(\epsilon)\otimes_R B$ is represented by $\epsilon
\otimes_R b$ for some $b \in B$; and $\epsilon\otimes_R b = 0$ if and
only if $b \in \mathfrak m_R B$.
\end{lem}

\begin{proof}
	Since, by the definition of a small extension,  the annihilator of $\epsilon$ is $\mathfrak m_R$, we have $(\epsilon) \iso R / \mathfrak m_R R$ and therefore $(\epsilon) \otimes_R B \iso B / \mathfrak m_R B$.
\end{proof}

\subsubsection{Applying the flatness criterion in our situation}

We define
$$\widehat A = \prod A/I_j.$$

\begin{lem}
	\label{L:3}
	If $N$ is a finitely presented $A$-module then the map
	\begin{equation*}
		N \mathop\otimes_A \prod A/I_j \to \prod (N \mathop\otimes_A A/I_j)
	\end{equation*}
	is an isomorphism.
\end{lem}
\begin{proof}
	If $N$ is finitely generated and free then $N \simeq A^n$ for some finite $n$ and then
	\begin{equation*}
		N \otimes \prod_j A/I_j = ( \prod_j A / I_j )^n
	\end{equation*}
	and
	\begin{equation*}
		\prod_j( N \otimes A/I_j) = \prod_j (A/I_j)^n,
	\end{equation*}
	giving the needed assertion.

	In general, $N$ admits a finite presentation
	\begin{equation*}
		A^r \to A^s \to N \to 0
	\end{equation*}
	and then we obtain a commutative diagram with exact rows:
	\begin{equation*}
		\xymatrix{
			A^r \otimes \prod A/I_j \ar[r] \ar[d]^\wr & A^s \otimes \prod A/I_j \ar[r] \ar[d]^\wr &  N \otimes \prod A/I_j \ar[r] \ar[d] & 0 \\
			\prod (A^r \otimes A/I_j) \ar[r] & \prod (A^s \otimes A/I_j) \ar[r] & \prod N \otimes A/I_j \ar[r] & 0
		}
	\end{equation*}
	The second row is exact because arbitrary products preserve cokernels.  We conclude by the 5-lemma that the rightmost vertical arrow is an isomorphism.
\end{proof}

\begin{lem}
\label{L:Ahat flat}
$\widehat A$ is flat over $R$.
\end{lem}

\begin{proof}
By Lemma \ref{L:flat criterion}, we need to show that the multiplication map $\mathfrak m_R\otimes_R \widehat A \to \widehat A$ is an injection. 
Because each $A/I_j$ is flat over $R$, we have inclusions
\[
  \xymatrix{
    \mathfrak m_R \otimes_R A/I_j \ar@{^(->}[r] & A/I_j
  }\]
  Taking the product, we obtain a commutative square:
\[
  \xymatrix{
	  \mathfrak m_R \otimes \widehat A \ar[r] \ar[d]_\wr & \widehat A \ar@{=}[d] \\
	  \prod (\mathfrak m_R \otimes_R A/I_j) \ar[r] & \prod A/I_j
  }
\]
where the vertical arrow on the left is an isomorphism due to \Cref{L:3}. 
The bottom arrow is injective because products preserve injections, so the top arrow must also be an injection.
\end{proof}

We will apply the following lemma with $B = A$ and $C = \widehat A$.

\begin{lem}
  \label{L:A flat}
	Let $R$ be an Artinian local ring and let $B \to C$ be a homomorphism of $R$-algebras.  Assume that $B \to C$ is injective and $C$ is flat over $R$.  Then the following are equivalent:
	\begin{enumerate}[label=(\alph*)]
		\item $B$ is flat over $R$.
		\item $a : B / \mathfrak m_R B \to C / \mathfrak m_R C$ is injective. \label{L:Aflat-p2}
		\item 
			$
			\ker\left(B \to C / \mathfrak m_R C \right) = \mathfrak m_RB.
  $
	\end{enumerate}
\end{lem}

\begin{proof}
	The third statement is a reformulation of the second.  We prove the equivalence of the first two.

	Suppose $R$ is a small extension of $R_0$ as in~\eqref{E:small}.  By \Cref{L:flat criterion}, $B$ is flat over $R$ if and only if $B \otimes_R R_0$ is flat over $R_0$ and the map $b$ in
\[
\xymatrix{
(\epsilon)\otimes_R B \ar[d]_c \ar[r]^-b & B \ar@{^(->}[d]\\
(\epsilon)\otimes_R C \ar@{^(->}[r] & C
}\]
	is injective.  Since $B \to C$ is injective and $(\epsilon) \otimes_R C \to C$ is injective, the injectivity of $b$ is equivalent to the injectivity of $c$.  But, using again that $(\epsilon)\cong R/\mathfrak m_R$, we have that  $(\epsilon) \otimes_R B \simeq B / \mathfrak m_R B$ and $(\epsilon) \otimes_R  C \simeq C / \mathfrak m_R C$, and so we see that $c$ corresponds to $a$ in assertion \eqref{L:Aflat-p2} of the lemma.

	Thus the flatness of $B$ over $R$ is equivalent to the flatness of $B \otimes_R R_0$ over $R_0$ and the injectivity of $a$.  But by induction, the flatness of $B \otimes_R R_0$ is equivalent to the injectivity of $a$.  Thus the flatness of $B$ over $R$ is equivalent to the injectivity of $a$.
\end{proof}

Since $A \to \widehat A$ is injective (because $\bigcap I_j = 0$) and $\widehat A$ is flat over $R$ (by \Cref{L:Ahat flat}), we conclude from \Cref{L:A flat} that $A$ is flat over $R$ if and only if $A/\mathfrak m_R A \to \widehat A / \mathfrak m_R \widehat A$ is injective.

\begin{lem}
\label{L:special reduced}
The special fiber $\widehat A / \mathfrak m_R \widehat A$ is reduced.
	If $A$ is flat over $R$ then $A / \mathfrak m_R A$ is also reduced.
\end{lem}

\begin{proof}
 	We start with the assertion that 
\begin{equation}\label{E:SpRed}
\widehat A / \mathfrak m_R \widehat A \cong \prod A/(I_j + \mathfrak m_R A).
\end{equation}
To establish this, our first observation is that $(A/I_j)\otimes_R(R/\mathfrak m_R)\cong A/(I_j+\mathfrak m_RA)$. To show this, recall
\ifHideFoot
\else 
\footnote{This is standard, but there is  a convenient reminder in Thm.~4.3 of Keith Conrad's note \url{https://kconrad.math.uconn.edu/blurbs/linmultialg/tensorprod.pdf}.}
\fi
  that 
if $B$ is a ring, and if $I$ and $J$ are ideals, then $B/(I+J) \cong (B/I) \otimes_B (B/J)$.
Then we have 
\begin{align*} \frac{A}{I_j + \mathfrak m_RA} &\cong \frac{A}{I_j} \otimes_A \frac{A}{\mathfrak m_RA} \\&\cong \frac{A}{I_j}\otimes_A(A\otimes_R \frac{R}{\mathfrak m_R})\\&\cong \frac{A}{I_j}\otimes_R\frac{R}{\mathfrak m_R}\\
& = (A/I_j)\otimes_R (R/\mathfrak  m_R).
\end{align*}
Taking products, we then have 
$$
 \prod A/(I_j + \mathfrak m_R A)\cong \prod ((A/I_j)\otimes_R (R/\mathfrak  m_R))\cong (\prod (A/I_j))\otimes_R (R/\mathfrak  m_R) \cong \widehat A/\mathfrak m_R\widehat A,
$$
where the second isomorphism comes from \Cref{L:3}.
	
	Having established \eqref{E:SpRed}, observe that 
	each of the factors in the product is reduced by hypothesis.  Therefore the product is reduced, and therefore $\widehat A $ is reduced.

	If in addition $A$ is flat over $R$, then, by \Cref{L:A flat}, the map
	\begin{equation*}
	\xymatrix{	A/ \mathfrak m_R A \ar[r]& \widehat A / \mathfrak m_R \widehat A}
	\end{equation*}
	is injective.  Since the target is reduced, this implies that $A / \mathfrak m_R A$ is reduced.
\end{proof}

This concludes the proof of Propositions~\ref{P:FlatRedFibAff} and~\ref{P:FlatRedFib}.

\subsection{Our earlier error}\label{S:error-expl}

We now take a moment to explain where an error arose in a previous version of this paper; this error led us to the \emph{erroneous} conclusion that the image of a morphism of abelian schemes was always an abelian scheme.

To understand the error, we first observe the following:

\begin{lem} \label{L:2}
  Let $B$ be a ring, and let $\st{{\mathfrak I}_1, \dots, {\mathfrak I}_N}$ be a collection of
  pairwise relatively prime ideals.  Then
  \begin{enumerate}[label=(\alph*)]
  \item ${\mathfrak I}_N$ is relatively prime to $\prod_{i<N}{\mathfrak I}_i$, and
  \item $\bigcap_{i\le N} {\mathfrak I}_i= \prod_{i\le N} {\mathfrak I}_i$.
  \end{enumerate}
  Moreover, for any ideal $\mathfrak J$ of $B$, 
\begin{equation}\label{E:PassTheCap}
    \bigcap_{i \le N} (\mathfrak J+{\mathfrak I}_i) = \mathfrak J + \bigcap_{i\le
      N} {\mathfrak I}_i.
  \end{equation}
             \end{lem}

\begin{proof}
  For (a), 
it suffices to prove that if $\mathfrak I_N$ is relatively prime to each of $\mathfrak I_{N-1},\dots,\mathfrak I_1$, then $\mathfrak I_N$ is relatively prime to $\prod_{j=1}^{N-1}\mathfrak I_j$.
By induction (or easily generalizing the argument for a direct proof), 
   it suffices to settle the case $N=3$.  We have
  \begin{align*}
    1 &= j_3+ j_1 = j_3'+j_2
        \intertext{where $j_i, j_i' \in {\mathfrak I}_i$ and so}
        1^2 =1 &= (j_3+j_1)(j_3'+j_2) \in {\mathfrak I}_3 + {\mathfrak I}_1{\mathfrak I}_2.
  \end{align*}

  Part (b) is the Chinese Remainder Theorem.

	For Part (c), let $C$ be the quotient ring $B/\mathfrak J$ and let $\varphi : B \to C$ be the quotient homomorphism.  We have
	\begin{equation*}
		\begin{aligned}
			\bigcap_{i \leq N} (\mathfrak J + \mathfrak I_i)
			& = \varphi^{-1} \Bigl( \bigcap_{i \leq N} \varphi(\mathfrak I_i) C \Bigr) \\
			& = \varphi^{-1} \Bigl( \prod_{i \leq N} \varphi(\mathfrak I_i) C \Bigr) & \qquad & \parbox{4cm}{by (b), since the $\varphi(\mathfrak I_i) C$ are pairwise coprime} \\
			& = \varphi^{-1} \Bigl( \varphi \bigl( \prod_{i \leq N} \mathfrak I_i \bigr) C \Bigr) \\
			& = \mathfrak J + \prod_{i \leq N} \mathfrak I_i \\
			& = \mathfrak J + \bigcap_{i \leq N} \mathfrak I_i & & \text{by (b).}
		\end{aligned}
	\end{equation*}
\end{proof}

	The identity~\eqref{E:PassTheCap} may fail for infinite intersections, but we always have one containment,
\begin{equation}\label{E:PassTheCapInf}
    \bigcap_{i } (\mathfrak J +{\mathfrak I}_i) \supseteq \mathfrak J  + \bigcap_{i} {\mathfrak I}_i,
  \end{equation} 
  since $ \mathfrak J + \mathfrak I_i\supseteq \mathfrak J + \bigcap \mathfrak I_i $ for all $i$.

	Note that the left side of~\eqref{E:PassTheCapInf} is the kernel of the map 
	\begin{equation*}
		B \to \prod B / ( \mathfrak J + \mathfrak I_i ) 
	\end{equation*}
	while the right side is the sum of $\mathfrak J$ and the kernel of
	\begin{equation*}
		B \to \prod  B / \mathfrak I_i .
	\end{equation*}

	In an earlier version of this paper, 
	 in trying to prove  \Cref{P:FlatRedFibAff}, 
	we assumed incorrectly that \eqref{E:PassTheCapInf} was an equality when $B = A$, $\mathfrak I_i = I_i$, and $\mathfrak J = \mathfrak m_R A$.  If this were the case, then it would imply that the kernel of
	\begin{equation*}
		A \to \widehat A / \mathfrak m_R \widehat A = \prod A / ( \mathfrak m_R + I_i)
	\end{equation*}
(see \eqref{E:SpRed} for the equality above) coincides with the sum of $\mathfrak m_R A$ and the kernel of
	\begin{equation*}
		A \to \widehat A = \prod A / I_i .
	\end{equation*}
	But $A \to \widehat A$ is injective (since by assumption $\bigcap I_i=0$), so this says that
	\begin{equation*}
		\mathfrak m_R = \ker \bigl( A \to \widehat A / \mathfrak m_R \widehat A \bigr),
	\end{equation*}
	which by \Cref{L:A flat} is equivalent the flatness of $A$ over $R$.  
		However, we have seen that $A$ can fail to be flat over $R$ (since otherwise the image of any morphism of abelian schemes would be an abelian scheme; see the proof of \Cref{T:ImIsAbSchArt}), and therefore the inclusion~\eqref{E:PassTheCapInf} must sometimes be strict.
	
	In other words, assuming that the containment \eqref{E:PassTheCapInf} is always an equality leads to the \emph{erroneous} conclusion that in \Cref{P:FlatRedFibAff}, the conditions \ref{P:FlatRedFibAff-red}, \ref{P:FlatRedFibAff-cop}, and \ref{P:FlatRedFibAff-dense} \emph{imply} that $A$ is flat.  This in turn leads to the \emph{erroneous} conclusion that in \Cref{P:FlatRedFib}, the conditions \ref{P:FlatRedFibAff-red}, \ref{P:FlatRedFibAff-cop}, and \ref{P:FlatRedFibAff-dense}  imply that $Z$ is flat over $S$; from there the proof of \Cref{T:ImIsAbSchArt} goes through as written, but without the necessary hypothesis that $f(X)$ be flat over $S$, leading to the \emph{erroneous} conclusion that the image of any homomorphism of abelian schemes is an abelian scheme.

 \bibliographystyle{amsalpha}
 \bibliography{DCG}

\end{document}